\documentclass[a4paper,12pt]{article}
\usepackage[utf8]{inputenc}

\usepackage{mathrsfs}
\usepackage{graphicx}
\usepackage{enumerate}
\usepackage{multicol}
\usepackage{color}
\usepackage{amsmath,amssymb,amscd}
\usepackage{amsthm}
\usepackage{mathabx}

\usepackage{pdfpages}
\usepackage{hyperref}

\usepackage{times}
\usepackage[varg]{txfonts}

\usepackage[top=1in, bottom=1in, left=0.5in, right=0.5in]{geometry}

\theoremstyle{plain}
\newtheorem{thm}{Theorem}
\newtheorem{lem}{Lemma}
\newtheorem{cor}{Corollary}
\newtheorem{prop}{Proposition}
\newtheorem{conj}{Conjecture}

\theoremstyle{definition}
\newtheorem{dfn}{Definition}
\newtheorem{ex}{Example}

\theoremstyle{remark}
\newtheorem{rem}{Remark}

\newtheorem*{ackn}{Acknowledgment}

\newcommand{\al}{\alpha}
\newcommand{\C}{\mathbb{C}}

\newcommand{\la}{\lambda}
\newcommand{\tla}{\widetilde{\lambda}}

\makeatletter
\newcommand{\subjclass}[2][1991]{%
  \let\@oldtitle\@title%
  \gdef\@title{\@oldtitle\footnotetext{#1 \emph{MSC classes.} #2}}%
}
\newcommand{\keywords}[1]{%
  \let\@@oldtitle\@title%
  \gdef\@title{\@@oldtitle\footnotetext{\emph{Keywords:} #1.}}%
}
\makeatother

\title{Polynomial equations for additive functions II.}
\author{Eszter Gselmann and Gergely Kiss}

\begin{document}

\maketitle


\begin{abstract}
In this sequence of work we investigate polynomial equations of additive functions. 
We consider the solutions of equation
\[
 \sum_{i=1}^{n}f_{i}(x^{p_{i}})g_{i}(x)^{q_{i}}= 0
 \qquad
 \left(x\in \mathbb{F}\right),
\]
where $n$ is a positive integer, $\mathbb{F}\subset \mathbb{C}$ is a field,
$f_{i}, g_{i}\colon \mathbb{F}\to \mathbb{C}$ are additive functions and $p_i, q_i$ are positive integers for all $i=1, \ldots, n$. Using the theory of decomposable functions we describe the solutions as compositions of higher order derivations and field homomorphisms. In many cases we also give a tight upper bound for the order of the involved derivations. Moreover, we present the full description of the solutions in some important special cases, too.
\end{abstract}

\section{Introduction}

As a continuation of our former work \cite{GseKis22a}, in this paper the additive solutions of a class of functional equations are studied. According to the results, this class of equations  turns out to be appropriate for characterizing homomorphisms and derivations, resp. acting between fields. The question how special morphisms (such as homomorphisms and derivations) can be characterized among additive mappings in general are important from algebraic point of view, but also from the perspective of functional equations.

Concerning all the cases we consider here, the involved additive functions are defined on a field $\mathbb{F}\subset \mathbb{C}$ and have values in the complex field, therefore
we introduce the preliminaries only in this setting.

In what follows, we adopt the standard notations, that is, $\mathbb{N}$ and $\mathbb{C}$ denote the set of positive integers and the set of complex numbers, respectively.

Henceforth assume $\mathbb{F}\subset \mathbb{C}$ to be a field.

\begin{dfn}
We say that a function $f\colon \mathbb{F}\to \mathbb{C}$ is \emph{additive} if it fulfills
\[
 f(x+y)= f(x)+f(y)
 \qquad
 \left(x, y\in \mathbb{F}\right).
\]
An additive function $d\colon \mathbb{F} \to \mathbb{C}$ is termed to be a \emph{derivation} if it also fulfills
\[
 d(xy)= d(x)y+xd(y)
 \qquad
 \left(x, y\in \mathbb{F}\right).
\]
An additive function $\varphi\colon\mathbb{F}\to \mathbb{C}$ is said to be a homomorphism if it is multiplicative as well, in other words, besides additivity we also have
\[
 \varphi(xy)= \varphi(x)\varphi(y)
 \qquad
 \left(x, y\in \mathbb{F}\right).
\]
If $\mathbb{F}= \mathbb{C}$ and $\varphi$ is an isomorphism, then $\varphi$ is called a \emph{complex automorphism}.
\end{dfn}

In this paper we investigate the solutions of
\begin{equation}\label{Eq_mixed}
 \sum_{i=1}^{n}f_{i}(x^{p_{i}})g_{i}(x)^{q_{i}}= 0,
 \qquad
 \left(x\in \mathbb{F}\right).
\end{equation}

Our primary aim is to show that the solutions can be represented with the aid of compositions of homomorphisms and (higher order) derivations. Using that we give the description of the solutions in many cases.

Observe that the above equation contains as a special case several well-known equations that characterize (higher order) derivations.
For instance, in \cite{Eba15, Eba17, GseKisVin18} the additive solutions of the equation
\[
\sum_{i=0}^{n}x^{i}f_{n+1-i}(x^{n+1-i})=0
\qquad
\left(x\in R\right)
\]
and also that of
\[
\sum_{i=0}^{n} f(x^{p_{i}})x^{q_{i}}=0
\qquad
\left(x\in R\right)
\]
were described on rings.

The core of the paper starts from the second section, where the results concerning \eqref{Eq_mixed} can be found. Firstly we prove some elementary yet important statements. The purpose of these lemmata is to figure out under what reasonable conditions we should make while studying these equations. For instance, the so-called Homogenization Lemma (see Lemma \ref{lemma_homogenization}) enables us to restrict ourselves to the case when for the parameters
\[
  p_{i}+q_{i}=N
  \qquad
  \left(i=1, \ldots, n\right)
 \]
hold.

Based on the remarks and examples that can be found at the beginning of Section \ref{SS2.1}, we will provide characterization theorems for equation \eqref{Eq_mixed} under the following conditions
\begin{enumerate}[{C}(i)]
 \item the positive integers $p_{1}, \ldots, p_{n}$ are arranged in a strictly increasing order, i.e., $p_1<\dots <p_{n}$;
 \item for all $i=1, \ldots, n$ we have $p_{i}+q_{i}=N$;
 \item for all $i, j \in \left\{ 1, \ldots, n\right\}$, $i\neq j$ we have $p_{i}\neq q_{j}$.
\end{enumerate}

Further, according to Lemma \ref{lem_equ_mixed}, the solutions of the above functional equations are sufficient to determine `up to equivalence'. This is because of the observation that if the functions $f_{1}, \ldots, f_{n}$ and $g_{1}, \ldots, g_{n}$ fulfill equation \eqref{Eq_mixed}, then for any automorphism $\varphi\colon \mathbb{C}\to \mathbb{C}$, the functions $\varphi \circ f_{1}, \ldots, \varphi \circ f_{n}$ and
$\varphi \circ g_{1},\ldots, \varphi \circ g_{n}$ also fulfill \eqref{Eq_mixed}.

Although equation \eqref{Eq_mixed} contains only one independent variable, we seek for the solutions in the class of additive functions. This property (i.e., additivity) will enable us to enlarge the number of independent variables in equation \eqref{Eq_mixed} from one to $N$. Briefly, this is the so-called Symmetrization method. After that, it is possible to prove that the functions $f_{1}, \ldots, f_{n}$ and $g_{1}, \ldots, g_{n}$ are decomposable functions on the multiplicative group $\mathbb{F}^{\times}$. From this we deduce that they are generalized exponential polynomials on the group $\mathbb{F}^{\times}$. Further, Theorem \ref{thm_mixed} says that for all $i=1, \ldots, n$, in the variety of the functions $f_{i}$ and $g_{i}$ there is exactly one exponential, namely $m_{i}$, provided that for the parameters $q_{1}, \ldots, q_{n}$ we assume $q_{i}< \dfrac{N}{2}$ for all $i=1, \ldots, n$. Our next aim is to get an upper bound for the degree of the involved higher order derivations that appear in the solutions. In connection to this, firstly we prove an alternative theorem, see Theorem \ref{T_mix_fo1}. After proving Theorem \ref{T_mix_fo1} for equation \eqref{Eq_mixed}, we  focus on some special cases. Here we consider equation \eqref{Eq_mixed} under the condition that all functions $f_i, g_i, \, i=1, \ldots, n$ are derivations or the linear functions and assume that
\begin{itemize}
    \item the order of the higher order derivations in the representations of the functions $f_{1}, \ldots, f_{n}$ is the same and we conclude that for all $i=1, \ldots, n$ we have $g_{i}(x)= \lambda_{i}x\; (x\in \mathbb{F})$, see Corollary \ref{cor7};
    \item the order of the higher order derivations in the representations of the functions $g_{1}, \ldots, g_{n}$ is the same and we conclude that for all $i=1, \ldots, n$ we have $g_{i}(x)= \lambda_{i}x\; (x\in \mathbb{F})$, see Corollary \ref{cor7};
    \item for all $i=1, \ldots, n$ we have $f_{i}(x)= c_{i} g_{i}(x)\; (x\in \mathbb{F})$ and we deduce that $f_{i}(x)= \lambda_{i} x\; (x\in \mathbb{F}, i=1, \ldots, n)$, see Corollary \ref{cor9}.
\end{itemize}
These results motivate Conjecture \ref{conj_mixed}, where we formulate that our conjecture is that the order of the higher order derivations in the representations of the functions $f_{1}, \ldots, f_{n}$ and $g_{1}, \ldots, g_{n}$ is at most $n-1$.
In a separate subsection, we close the last section with some special cases of equation \eqref{Eq_mixed}.

As we wrote above, this paper can be considered as a natural continuation or completion of our former one \cite{GseKis22a}. In that paper we considered in the same situation a rather similar equation, namely
\[
\tag{$\ast$}
\sum_{i=1}^{n}f_{i}(x^{p_{i}})g_{i}(x^{q_{i}})=0
\qquad
\left(x\in \mathbb{F}\right),
\]
where where $n$ is a positive integer, $\mathbb{F}\subset \mathbb{C}$ is a field and $f_{i}, g_{i}\colon \mathbb{F}\to \mathbb{C}$ are additive functions for all $i=1, \ldots, n$.

Besides equations \eqref{Eq_mixed} and $(\ast)$, the equation
\[
\tag{$\diamond$}
\sum_{i=1}^{n}f_{i}(x^{p_{i}})g_{i}(x)^{q_{i}}=0
\qquad
\left(x\in \mathbb{F}\right),
\]
can also be considered.
Unfortunately, the similar statements as for equations \eqref{Eq_mixed} and $(\ast)$,  \emph{are not satisfied} for equation $(\diamond)$, even though with assumptions C(i)--C(iii). After symmetrization we can deduce that the involved additive functions are linearly dependent. At the same time, in case of equation $(\diamond)$, \emph{arbitrary} additive functions can appear in the solution space. More precisely, let $f\colon \mathbb{F}\to \mathbb{C}$ be an \emph{arbitrary} additive function and suppose that  the complex constants $\lambda_{1}, \ldots, \lambda_{n}$ and $\mu_{1}, \ldots, \mu_{n}$ fulfill
\[
 \sum_{i=1}^{n}\lambda_{i}^{p_{i}}\mu_{i}^{q_{i}}=0,
\]
then the functions
\[
 f_{i}(x)= \lambda_{i}f(x)
 \qquad
 \text{and}
 \qquad
 g_{i}(x)= \mu_{i}f(x)
 \qquad
 \left(x\in \mathbb{F}\right)
\]
fulfill equation $(\diamond)$.

Obviously, not only the structure of these equations, but also the setup is rather similar. This is mainly due to the fact that these problems are special cases of a much more general problem. More concretely, the problem-raising, the solution methods of this paper and also that of the previous ones in \cite{Eba15, Eba17, EbaRieSah, GseKisVin18, GseKisVin19, GseKis22a} can be regarded as initial steps for the general problem below. Let $n$ be a positive integer, $P$ be a given multivariate complex polynomial and $Q_{1}, \ldots, Q_{n}\colon \mathbb{F}\to \mathbb{C}$ be polynomials. What can be said about the additive functions $f_{1}, \ldots, f_{n}\colon \mathbb{F}\to \mathbb{C}$ if they fulfill equation
\[
\tag{$\bullet$}
P(f_{1}(Q_{1}(x)), \ldots, f_{n}(Q_{n}(x)))=0
\]
for all $x\in \mathbb{F}$.
In connection with equation $(\bullet)$, it is obviously necessary to first clarify under what additional conditions it is expected that we can state something more about the unknown functions involved (obviously beyond additivity). Further, one of the heaviest difficulties with this equation is that it contains only one independent variable (this typically carries little information), but the number of unknown functions can be large. However, additivity provides the opportunity to increase the number of independent variables. These are done typically through proving Symmetrization lemmata. After that, our goal is usually to show that the involved additive functions have a `good connection' with the multiplicative structure, as well. More concretely, our objective is to show that these additive functions are decomposable functions on the multiplicative group $\mathbb{F}^{\times}$. This is the first point where (at least at the level of the proofs) it becomes clear that equations $(\ast)$ and \eqref{Eq_mixed} cannot be handled with the same method. Observe that compared to equation $(\ast)$, in equation \eqref{Eq_mixed} the role of the parameters is \emph{not equal}. It is clear from the comparison of the methods of \cite{GseKis22a} and this paper that this fact (the role of the parameters) has both advantages and disadvantages.

The asymmetry in the parameters makes it possible the statement of Lemma \ref{lem_decop_mixed} to hold under more general conditions (during the proof it is enough only to use that the parameters $p_{1}, \ldots, p_{n}$ are different.) At the same time, in case of equation $(\ast)$ it is possible to verify in one step that the functions $f_{1}, \ldots, f_{n}$ and $g_{1}, \ldots, g_{n}$ are decomposable functions on $\mathbb{F}^{\times}$. In case of equation \eqref{Eq_mixed} first we can only show that the functions $f_{1}, \ldots, f_{n}$ are decomposable (cf. Lemma \ref{lem_decop_mixed}). Based on this, finally the decomposability of the functions $g_{1}, \ldots, g_{n}$ can also be deduced, see Theorem \ref{thm_decop_mixed2}.
This and also the subsequent results show that in the case of equation \eqref{Eq_mixed}, new solution methods were necessary to develop.

\section{Reduction of the problem}\label{SS2.1}

This part begins with some elementary, yet fundamental observations. Our aim here is to show that in any cases the original problem can be reduced to a more simpler equation. Here we follow the monograph \cite{Sze91}. Although the proofs of Lemmata \ref{lemma_homogenization} and \ref{lem_sym_mixed} are analogous to that of Lemmata 4 and 6 of \cite{GseKis22a}, for the sake of completeness, we present their proofs 
in the Appendix located  at the end of this work.

\begin{dfn}
 Let $G, S$ be commutative semigroups (written additively), $n\in \mathbb{N}$ and let $A\colon G^{n}\to S$ be a function.
 We say that $A$ is \emph{$n$-additive} if it is a homomorphism of $G$ into $S$ in each variable.
 If $n=1$ or $n=2$ then the function $A$ is simply termed to be \emph{additive}
 or \emph{biadditive}, respectively.
\end{dfn}

The \emph{diagonalization} or \emph{trace} of an $n$-additive
function $A\colon G^{n}\to S$ is defined as
 \[
  A^{\ast}(x)=A\left(x, \ldots, x\right)
  \qquad
  \left(x\in G\right).
 \]
As a direct consequence of the definition each $n$-additive function
$A\colon G^{n}\to S$ satisfies
\[
 A(x_{1}, \ldots, x_{i-1}, kx_{i}, x_{i+1}, \ldots, x_n)
 =
 kA(x_{1}, \ldots, x_{i-1}, x_{i}, x_{i+1}, \ldots, x_{n})
 \qquad
 \left(x_{1}, \ldots, x_{n}\in G\right)
\]
for all $i=1, \ldots, n$, where $k\in \mathbb{N}$ is arbitrary. The
same identity holds for any $k\in \mathbb{Z}$ provided that $G$ and
$S$ are groups, and for $k\in \mathbb{Q}$, provided that $G$ and $S$
are linear spaces over the rationals. This immediately implies that for the diagonalization of $A$
we have
\[
 A^{\ast}(kx)=k^{n}A^{\ast}(x)
 \qquad
 \left(x\in G\right).
\]

The above notion can also be extended for the case $n=0$ by letting
$G^{0}=G$ and by calling $0$-additive any constant function from $G$ to $S$.

\medskip

Based on the above notions and results our first lemma can be proved.

\begin{lem}[Homogenization]\label{lemma_homogenization}
 Let $n$ be a positive integer, $\mathbb{F}\subset \mathbb{C}$ be a field and
 $p_{1}, \ldots, p_{n}, q_{1}, \ldots, q_{n}$ be fixed positive integers.
 Assume that the additive functions $f_{1}, \ldots, f_{n}, g_{1}, \ldots, g_{n}\colon \mathbb{F}\to \mathbb{C}$ satisfy functional equation \eqref{Eq_mixed}, i.e.,
  \begin{equation*}
\sum_{i=1}^{n}f_{i}(x^{p_{i}})g_{i}(x)^{q_{i}}= 0
 \end{equation*}
 for each $x\in \mathbb{F}$.
 If the set $\left\{ p_{1}, \ldots, p_{n}\right\}$ has a partition $\mathcal{P}_{1}, \ldots, \mathcal{P}_{k}$ with the property
 \[
 \text{if } p_{\alpha}, p_{\beta} \in \mathcal{P}_{j} \text{ for a certain index $j$, then } p_{\alpha}+q_{\alpha}= p_{\beta}+q_{\beta},
 \]
then the system of equations
\[
 \sum_{p_{\alpha}\in \mathcal{P}_{j}} f_{\alpha}(x^{p_{\alpha}})g_{\alpha}(x))^{q_{\alpha}}=0
 \qquad
 \left(x\in \mathbb{F}, j=1, \ldots, k\right)
\]
is satisfied.
\end{lem}

\begin{rem}\label{rem2.1}
 The above lemma guarantees that \emph{ab initio}
 \[
  p_{i}+q_{i}=N
  \qquad
  \left(i=1, \ldots, n\right)
 \]
can be assumed. Otherwise, after using the above homogenization, we get a system of functional equations in which this condition is already fulfilled.
For instance, due to the above lemma, if the additive functions $f_{1}, \ldots, f_{5}\colon \mathbb{F}\to \mathbb{C}$ and $g_{1}, \ldots, g_{5}\colon \mathbb{F}\to \mathbb{C}$ satisfy equation
\[ f_{1}(x^{16})g_{1}(x)^{5}+f_{2}(x^{12})g_{2}(x)^{9}+f_{3}(x^{11})g_{3}(x)^{10}
+f_{4}(x^{3})g_{4}(x)^{7}+ f_{5}(x^{2})g_{4}(x)^{8}
 =
 0
 \qquad
 \left(x\in \mathbb{F}\right)\]
then
the equations
\[
 f_{1}(x^{16})g_{1}(x)^{5}+f_{2}(x^{12})g_{2}(x)^{9}+f_{3}(x^{11})g_{3}(x)^{10}=0
 \qquad
 \left(x\in \mathbb{F}\right)
 \]
and
\[
   f_{4}(x^{3})g_{4}(x)^{7}+ f_{5}(x^{2})g_{4}(x)^{8}=0
   \qquad
   \left(x\in \mathbb{F}\right)
\]
are also fulfilled (separately).
\end{rem}

\begin{rem}\label{rem2.3}
 At first glance the assumption that $p_{1}, \ldots, p_{n}$ are different seems a reasonable and sufficient supposition.
 Clearly, if the parameters $p_i$ are not necessarily different, then we cannot expect anything special for the form of the involved additive functions.
 To see this, let $p$ and $q$ be positive integers and $f\colon \mathbb{F}\to \mathbb{C}$ be an \emph{arbitrary} additive function, $\lambda$ be a complex number such that $1+\lambda^{q}=0$ and
 define the complex-valued functions $f_{1}, g_{1}, f_{2}, g_{2}$ on $\mathbb{F}$ by
 \[
  f_{1}(x)= f(x) \quad
  g_{1}(x)= f(x) \quad
  f_{2}(x)= f(x) \quad
  g_{2}(x)= \lambda \cdot f(x)
  \qquad
  \left(x\in \mathbb{F}\right).
 \]
An immediate computation shows that we have
\[
 f_{1}(x^{p})g_{1}(x)^{q}+f_{2}(x^{p})g_{2}(x)^{q}= 0
 \qquad
 \left(x\in \mathbb{F}\right).
\]
Note however, that $p_i=q_j$ for some $i,j\in \{1, \dots, n\}$ can be handled, that we emphasis at some points of this paper. On the other hand, to avoid further difficulties throughout of the work we simple assume that $p_1,\dots p_n, q_1, \dots, q_n$ are distinct positive integers.
\end{rem}

In view of the above remarks, from now on, the following assumptions are adopted.

\begin{enumerate}[{C}(i)]
 \item the positive integers $p_{1}, \ldots, p_{n}$ are arranged in a strictly increasing order, i.e., $p_1<\dots <p_{n}$;
 \item for all $i=1, \ldots, n$ we have $p_{i}+q_{i}=N$;
 \item for all $i, j \in \left\{ 1, \ldots, n\right\}$, $i\neq j$ we have $p_{i}\neq q_{j}$.
\end{enumerate}

\begin{rem}\label{rem2.4}
 Define the relation $\sim$ on $\mathbb{F}^{\mathbb{C}}$ by \[
  f\sim g \quad
  \text{if and only if there exists an automorphism $\varphi \colon \mathbb{C}\to \mathbb{C}$
  such that $\varphi \circ f=g$. }
 \]
Obviously $\sim$ is an equivalence relation on $\mathbb{F}^{\mathbb{C}}$ that induces a partition on $\mathbb{F}^{\mathbb{C}}$.
\end{rem}

\begin{lem}[Equivalence]\label{lem_equ_mixed}
 Let $n$ be a positive integer, $\mathbb{F}\subset \mathbb{C}$ be a field and
 $p_{1}, \ldots, p_{n}, q_{1}, \ldots, q_{n}$ be fixed positive integers fulfilling the conditions C(i)-C(iii) of Remark \ref{rem2.3}.
 Assume that the additive functions $f_{1}, \ldots, f_{n},$ $g_{1}, \ldots, g_{n}\colon \mathbb{F}\to \mathbb{C}$ satisfy functional equation
 \eqref{Eq_mixed}, that is, we have
 \[
 \sum_{i=1}^{n}f_{i}(x^{p_{i}})g(x)^{q_{i}}=0
 \]
 for all $x\in \mathbb{F}$.
 Then for an arbitrary automorphism $\varphi \colon \mathbb{C}\to \mathbb{C}$ the functions
 $\varphi \circ f_{1}, \ldots, \varphi \circ f_{n}, \varphi \circ g_{1}, \ldots, \varphi \circ g_{n}$ also fulfill equation \eqref{Eq_mixed}.
\end{lem}

\begin{rem}
 We can always restrict ourselves to the case when all the involved functions are non-identically zero. Otherwise, the number of the terms
 appearing in equation \eqref{Eq_mixed} can be reduced.
\end{rem}
\medskip

One of the most important theoretical results concerning
multiadditive functions is the so-called \emph{Polarization
formula}, that briefly expresses that every $n$-additive symmetric
function is \emph{uniquely} determined by its diagonalization under
some conditions on the domain as well as on the range. Suppose that
$G$ is a commutative semigroup and $S$ is a commutative group. The
action of the {\emph{difference operator}} $\Delta$ on a function
$f\colon G\to S$ is defined by the formula
\[\Delta_y f(x)=f(x+y)-f(x)
\qquad
\left(x, y\in G\right). \]
Note that the addition in the argument of the function is the
operation of the semigroup $G$ and the subtraction means the inverse
of the operation of the group $S$.
The superposition of several difference operators will be denoted shortly
\[
\Delta_{y_1 \ldots y_{n}} f = \Delta_{y_{1}} \Delta_{y_{2}} \ldots \Delta_{y_{n}} f
\qquad
\left( n \in\mathbb{N}\right).
 \]

\begin{thm}[Polarization formula]\label{Thm_polarization}
 Suppose that $G$ is a commutative semigroup, $S$ is a commutative group, $n\in \mathbb{N}$.
 If $A\colon G^{n}\to S$ is a symmetric, $n$-additive function, then for all
 $x, y_{1}, \ldots, y_{m}\in G$ we have
 \[
  \Delta_{y_{1}, \ldots, y_{m}}A^{\ast}(x)=
  \left\{
  \begin{array}{rcl}
   0 & \text{ if} & m>n \\
   n!A(y_{1}, \ldots, y_{m}) & \text{ if}& m=n.
  \end{array}
  \right.
 \]

\end{thm}


\begin{lem}
\label{mainfact}
  Let $n\in \mathbb{N}$ and suppose that the multiplication by $n!$ is surjective in the commutative semigroup $G$ or injective in the commutative group $S$. Then for any symmetric, $n$-additive function $A\colon G^{n}\to S$, $A^{\ast}\equiv 0$ implies that
  $A$ is identically zero, as well.
\end{lem}

\begin{dfn}
 Let $G$ and $S$ be commutative semigroups, a function $p\colon G\to S$ is called a \emph{generalized polynomial} from $G$ to $S$, if it has a representation as the sum of diagonalization of symmetric multi-additive functions from $G$ to $S$. In other words, a function $p\colon G\to S$ is a
 generalized polynomial if and only if, it has a representation
 \[
  p= \sum_{k=0}^{n}A^{\ast}_{n},
 \]
where $n$ is a nonnegative integer and $A_{k}\colon G^{k}\to S$ is a symmetric, $k$-additive function for each
$k=0, 1, \ldots, n$. In this case we also say that $p$ is a generalized polynomial \emph{of degree at most $n$}.

Let $n$ be a nonnegative integer, functions $p_{n}\colon G\to S$ of the form
\[
 p_{n}= A_{n}^{\ast},
\]
where $A_{n}\colon G^{n}\to S$ are the so-called \emph{generalized monomials of degree $n$}.
\end{dfn}

In this subsection $(G, \cdot)$ is assumed to be  a commutative group (written multiplicatively).

During the proof of Lemma \ref{lem_sym_mixed} we use the following lemma from \cite{GseKis22a}.

\begin{lem}\label{lem_monom}
 Let $k$ and $n$ be positive integers, $\mathbb{F}\subset \mathbb{C}$ be a field and
 $m_{1}, \ldots, m_{n}\colon \mathbb{F}\to \mathbb{C}$ be generalized monomials of degree $k$.
 If
 \[
  \sum_{i=1}^{n}m_{i}(x)=0
 \]
holds for all $x\in \mathbb{F}$, then
\[
 \sum_{i=1}^{n}M_{i}(x_{1}, \ldots, x_{k})=0
\]
is fulfilled for all $x_{1}, \ldots, x_{k}$, where for all $i=1, \ldots, n$, the mapping
$M_{i}\colon \mathbb{F}^{k}\to \mathbb{C}$ is the uniquely determined symmetric, $k$-additive function such that
\[
 M_{i}(x, \ldots, x)= m_{i}(x)
 \qquad
 \left(x\in \mathbb{F}\right).
\]
\end{lem}

\begin{lem}[Symmetrization]\label{lem_sym_mixed}
 Let $n$ be a positive integer, $\mathbb{F}\subset \mathbb{C}$ be a field and
 $p_{1}, \ldots, p_{n}, q_{1}, \ldots, q_{n}$ be fixed positive integers fulfilling conditions C(ii).
Assume that the additive functions $f_{1}, \ldots, f_{n}, g_{1}, \ldots, g_{n}\colon \mathbb{F}\to \mathbb{C}$ satisfy functional equation
 \eqref{Eq_mixed}, i.e.,
 \[
 \sum_{i=1}^{n}f_{i}(x^{p_{i}})g(x)^{q_{i}}=0
 \]
 holds  for each $x\in \mathbb{F}$. Then
 \[
  \frac{1}{N!} \sum_{\sigma \in \mathscr{S}_{N}}\sum_{i=1}^{n} f_{i}(x_{\sigma(1)} \cdots x_{\sigma(p_{i})}) \cdot g_{i}(x_{\sigma(p_{i}+1)}) \cdots g_{i}(x_{\sigma(N)})=0
 \]
holds for all $x_{1}, \ldots, x_{N}\in \mathbb{F}$.
\end{lem}

\section{Preliminary results}
Decomposable functions will play a key role in the sequel. This notion was introduced by E.~Shulman in \cite{Shu10}.
Besides this we heavily rely on the work of Laczkovich \cite{Lac19}.

\begin{dfn}
 Let $G$ be a group and $n\in \mathbb{N}, n\geq 2$.
 A function $F\colon G^{n}\to \mathbb{C}$ is said to be
 \emph{decomposable} if it can be written as a finite sum of products
 $F_{1}\cdots F_{k}$, where all $F_{i}$ depend on disjoint sets of variables.
\end{dfn}

\begin{rem}
 Without loss of generality we can suppose that $k=2$ in the above definition, that is,
 decomposable functions are those mappings that can be written in the form
 \[
  F(x_{1}, \ldots, x_{n})= \sum_{E}\sum_{j}A_{j}^{E}B_{j}^{E}
 \]
where $E$ runs through all non-void proper subsets of $\left\{1,
\ldots, n\right\}$ and for each $E$ and $j$ the function $A_{j}^{E}$
depends only on variables $x_{i}$ with $i\in E$, while $B_{j}^{E}$
depends only on the variables $x_{i}$ with $i\notin E$.
\end{rem}

The theorem below is about the connection between decomposable functions and generalized exponential polynomials, see Laczkovich \cite{Lac19}.

\begin{thm}\label{thm_Laczk}
Let $G$ be a commutative topological  semigroup with unit.
A continuous function $f\colon G\to \mathbb{C}$ is a generalized exponential polynomial
\emph{if and only if} there is a positive integer $n\geq 2$ such that the mapping
\[
G^{n} \ni (x_{1}, \ldots, x_{n}) \longmapsto f(x_1+\cdots+ x_n)
\]
is decomposable.
\end{thm}

Now we show that the functions $f_{1}, \ldots, f_{n}$ are decomposable functions. This together with equation \eqref{Eq_mixed} will yield that the functions $g_{1}, \ldots, g_{n}$ are decomposable functions, too. After this, we apply Theorem \ref{thm_Laczk}, that immediately yield that the solutions of equation \eqref{Eq_mixed} are generalized exponential polynomials of the multiplicative group $\mathbb{F}^{\times}$.

\begin{lem}\label{lem_decop_mixed}
 Let $n$ be a positive integer, $\mathbb{F}\subset \mathbb{C}$ be a field and
 $p_{1}, \ldots, p_{n}, q_{1}, \ldots, q_{n}$ be fixed positive integers fulfilling conditions C(i)--C(iii).
 Assume that the additive functions $f_{1}, \ldots, f_{n}, g_{1}, \ldots, g_{n}\colon \mathbb{F}\to \mathbb{C}$ satisfy functional equation
 \eqref{Eq_mixed}, that is,
 \[
 \sum_{i=1}^{n}f_{i}(x^{p_{i}})g_{i}(x)^{q_{i}}=0
 \]
for each $x\in \mathbb{F}$. Then all the functions $f_{1}, \ldots, f_{n}$ are decomposable functions of the group $\mathbb{F}^{\times}$.
\end{lem}

\begin{proof}
 Let $n$ be a positive integer, $\mathbb{F}\subset \mathbb{C}$ be a field and
 $p_{1}, \ldots, p_{n}, q_{1}, \ldots, q_{n}$ be fixed positive integers fulfilling conditions C(i)--C(iii).
 Assume that the additive functions $f_{1}, \ldots, f_{n}, g_{1}, \ldots, g_{n}\colon \mathbb{F}\to \mathbb{C}$ satisfy functional equation
 \eqref{Eq_mixed}
 for each $x\in \mathbb{F}$.
 Let
 \[
  S= \left\{p_{1}, \ldots, p_{n} \right\}
 \]
Then  due to condition C(i) $\max S= p_{n}$. In view of  Lemma \ref{lem_sym_mixed}, we have
 \[
  \frac{1}{N!} \sum_{\sigma \in \mathscr{S}_{N}}\sum_{i=1}^{n} f_{i}(x_{\sigma(1)} \cdots x_{\sigma(p_{i})}) \cdot g_{i}(x_{\sigma(p_{i}+1)}) \cdots g_{i}(x_{\sigma(N)})=0
 \]
for all $x_{1}, \ldots, x_{N}\in \mathbb{F}$, or after some rearrangement,
\begin{multline*}
   \frac{1}{N!} \sum_{\sigma \in \mathscr{S}_{N}} f_{n}(x_{\sigma(1)} \cdots x_{\sigma(p_{n})}) \cdot g_{n}(x_{\sigma(p_{n}+1)}) \cdots g_{n}(x_{\sigma(N)})
   \\=
   -\frac{1}{N!} \sum_{\sigma \in \mathscr{S}_{N}}\sum_{i=1}^{n-1} f_{i}(x_{\sigma(1)} \cdots x_{\sigma(p_{i})}) \cdot g_{i}(x_{\sigma(p_{i}+1)}) \cdots g_{i}(x_{\sigma(N)})
   \qquad
   \left(x_{1}, \ldots, x_{N}\in \mathbb{F}^{\times}\right).
   \end{multline*}
Let now
\[
 x_{p_{n}+1}= \cdots = x_{N}= 1,
\]
then the above identity says that $g_{n}(1)^{q_{1}} \cdot f_{n}$ is decomposable. If $g_{n}(1)$ were zero, but $g_{n}$ would not be identically zero, then the would exists
$a\in \mathbb{F}^{\times}$ such that
$g_{n}(a)\neq 0$. In this case the above substitutions should be modified to
\[
 x_{p_{n}+1}= a , \; x_{p_{n}+2}= \cdots = x_{N}= 1,
\]
to get the same conclusion.

After that, let us consider the set $S\setminus \left\{ p_{n}\right\}$ and apply the above argument for this set.
 With this step-by-step descending argument follows the statement of the lemma.
 \end{proof}

To verify that $g_{1}, \ldots, g_{n}$ are also decomposable functions we have to introduce the notions of  exponential polynomials.

\medskip


\begin{dfn}
{\it Polynomials} are elements of the algebra generated by additive
functions over $G$. Namely, if $n$ is a positive integer,
$P\colon\mathbb{C}^{n}\to \mathbb{C}$ is a (classical) complex
polynomial in
 $n$ variables and $a_{k}\colon G\to \mathbb{C}\; (k=1, \ldots, n)$ are additive functions, then the function
 \[
  x\longmapsto P(a_{1}(x), \ldots, a_{n}(x))
 \]
is a polynomial and, also conversely, every polynomial can be
represented in such a form.
\end{dfn}

\begin{rem}
 We recall that the elements of $\mathbb{N}^{n}$ for any positive integer $n$ are called
 ($n$-dimensional) \emph{multi-indices}.
 Addition, multiplication and inequalities between multi-indices of the same dimension are defined component-wise.
 Further, we define $x^{\alpha}$ for any $n$-dimensional multi-index $\alpha$ and for any
 $x=(x_{1}, \ldots, x_{n})$ in $\mathbb{C}^{n}$ by
 \[
  x^{\alpha}=\prod_{i=1}^{n}x_{i}^{\alpha_{i}}
 \]
where we always adopt the convention $0^{0}=0$. We also use the
notation $\left|\alpha\right|= \alpha_{1}+\cdots+\alpha_{n}$. With
these notations any polynomial of degree at most $N$ on the
commutative semigroup $G$ has the form
\[
 p(x)= \sum_{\left|\alpha\right|\leq N}c_{\alpha}a(x)^{\alpha}
 \qquad
 \left(x\in G\right),
\]
where $c_{\alpha}\in \mathbb{C}$ and $a=(a_1, \dots, a_n) \colon
G\to \mathbb{C}^{n}$ is an additive function. Furthermore, the
\emph{homogeneous term of degree $k$} of $p$ is
\[
 \sum_{\left|\alpha\right|=k}c_{\alpha}a(x)^{\alpha} .
\]
\end{rem}

\begin{lem}[Lemma 2.7 of \cite{Sze91}]\label{L_lin_dep}
 Let $G$ be a commutative group,
 $n$ be a positive integer and let
 \[
  a=\left(a_{1}, \ldots, a_{n}\right),
 \]
where $a_{1}, \ldots, a_{n}$ are linearly independent complex valued
additive functions defined on $G$. Then the monomials
$\left\{a^{\alpha}\right\}$ for different multi-indices are linearly
independent.
\end{lem}

\begin{dfn}
A function $m\colon G\to \mathbb{C}$ is called an \emph{exponential}
function if it satisfies
\[
 m(xy)=m(x)m(y)
 \qquad
 \left(x,y\in G\right).
\]
Furthermore, on a(n)  \emph{(generalized) exponential polynomial} we mean a linear
combination of functions of the form $p \cdot m$, where $p$ is a
(generalized) polynomial and $m$ is an exponential function.
\end{dfn}

The following lemma shows that generalized exponential polynomial functions are linearly independent. Although it can be stated more generally (see \cite{Sze91}), we adopt it to our situation, when the functions are complex valued.
\begin{lem}[Lemma 4.3 of \cite{Sze91}]\label{L_Lin_Ind}
 Let $G$ be a commutative group, $n$ a positive integer, 
 $m_{1}, \ldots, m_{n} \colon G\to \mathbb{C}\, (i=1, \ldots, n)$ be distinct nonzero exponentials and  $p_{1}, \ldots, p_{n} \colon  G\to \mathbb{K}\, (i=1, \ldots, n)$ be generalized polynomials. If  $\displaystyle \sum_{i=1}^n p_i\cdot m_i$ is identically zero, then for all $i=1, \ldots, n$ the generalized polynomial  $p_i$ is identically zero.
\end{lem}

However we will need the analogue statement for polynomial expressions of generalized exponential polynomials which was proved in \cite{GseKisVin18}.

\begin{thm}\label{T_Poly_Ind}
Let $\mathbb{K}$ be a field of characteristic $0$ and $k,l,N$ be
positive integers such that $k,l\le N$. Let $m_1, \dots, m_k\colon
\mathbb{K}^{\times}\to \mathbb{C}$ be distinct exponential functions
that are additive on $\mathbb{K}$, let $a_1, \dots, a_l\colon
\mathbb{K}^{\times}\to \mathbb{C}$ be additive functions that are
linearly independent over $\mathbb{C}$ and  for all $|s|\leq N$ let $ P_s\colon
\mathbb{C}^l\to \mathbb{C}$ be classical complex polynomials of $l$
variables. If
\[
    \sum_{|s|\le N } P_{s}(a_1, \dots, a_l) m_1^{s_1}\cdots m_k^{s_k}=0
\]
then for all $|s|\leq N$, the polynomials $P_s$ vanish identically.
\end{thm}



It is easy to see that
each polynomial, that is, any function of the form
\[
  x\longmapsto P(a_{1}(x), \ldots, a_{n}(x)),
 \]
where $n$ is a positive integer,
$P\colon\mathbb{C}^{n}\to \mathbb{C}$ is a (classical) complex
polynomial in
$n$ variables and $a_{k}\colon G\to \mathbb{C}\; (k=1, \ldots, n)$ are additive functions, is a generalized polynomial.
The converse however is in general not true. A complex-valued generalized polynomial $p$ defined on a commutative group $G$ is a
polynomial \emph{if and only if} its variety (the linear space spanned by its translates) is of \emph{finite} dimension.


The notion of derivations can be extended in several ways. We will employ the concept of higher order derivations according to Reich \cite{Rei98} and Unger--Reich \cite{UngRei98}. For further results on characterization theorems on higher order derivations consult e.g. \cite{Eba15, Eba17, EbaRieSah} and
\cite{GseKisVin18}.

\begin{dfn}
 Let $\mathbb{F}\subset \mathbb{C}$ be a field. The identically zero map is the only \emph{derivation of order zero}. For each $n\in \mathbb{N}$, an additive mapping
 $f\colon \mathbb{F}\to \mathbb{C}$ is termed to be a \emph{derivation of order $n$}, if there exists $B\colon \mathbb{F}\times \mathbb{F}\to \mathbb{C}$ such that
 $B$ is a bi-derivation of order $n-1$ (that is, $B$ is a derivation of order $n-1$ in each variable) and
 \[
  f(xy)-xf(y)-f(x)y=B(x, y)
  \qquad
  \left(x, y\in \mathbb{F}\right).
 \]
 The set of derivations of order $n$ of the ring $R$ will be denoted by $\mathscr{D}_{n}(\mathbb{F})$.
\end{dfn}

\begin{rem}
\label{pathologic}
Since $\mathscr{D}_{0}(\mathbb{F})=\left\{0\right\}$, the only bi-derivation of order zero is the identically zero function, thus $f\in \mathscr{D}_{1}(\mathbb{F})$ if and only if
  \[
   f(xy)=xf(y)+f(x)y
   \qquad
   \left(x, y\in \mathbb{F}\right),
  \]
that is, the notions of first order derivations and derivations coincide. On the other hand for any $n\in \mathbb{N}$ the set $\mathscr{D}_{n}(\mathbb{F})\setminus \mathscr{D}_{n-1}(\mathbb{F})$ is nonempty because  $d_{1}\circ \cdots \circ d_{n}\in \mathscr{D}_{n}(\mathbb{F})$, but $d_{1}\circ \cdots \circ d_{n}\notin \mathscr{D}_{n-1}(R)$, where $d_{1}, \ldots, d_{n}\in \mathscr{D}_{1}(\mathbb{F})$ are non-identically zero derivations.
\end{rem}

For our future purposes the notion of differential operators will also be important, see \cite{KisLac18}.

\begin{dfn}
 Let $\mathbb{F}\subset \mathbb{C}$ be a field. We say that the map
 $D \colon \mathbb{F}\to \mathbb{C}$ is a \emph{differential operator of degree at most $n$} if $D$ is the linear combination, with coefficients from $\mathbb{F}$, of finitely many maps of the form
 $d_1 \circ \cdots \circ d_k$, where $d_1, \ldots, d_k$ are derivations on $\mathbb{F}$ and $k\leq n$. If $k = 0$ then we interpret $d_1\circ \cdots \circ d_k$ as the identity function.
 We denote by $\mathscr{O}_n(\mathbb{F})$ the set of differential operators of degree at most $n$ defined on $\mathbb{F}$. We say that the degree of a differential operator $D$ is
$n$ if $D \in \mathscr{O}_{n}(\mathbb{F})\setminus\mathscr{O}_{n-1}(\mathbb{F})$ (where $\mathscr{O}_{-1}(\mathbb{F})= \emptyset$, by definition).
\end{dfn}


The main result of \cite{KisLac18} is Theorem 1.1 that reads in our settings as follows.

\begin{thm}\label{thm_KisLac}
 Let $\mathbb{F}\subset \mathbb{C}$ be a field and let $n$ be a positive integer. Then, for every function $D \colon \mathbb{F}\to \mathbb{C}$, the
following are equivalent.
\begin{enumerate}[(i)]
\item $D\in \mathscr{D}_{n}(\mathbb{F})$
\item $D\in \mathrm{cl}\left(\mathscr{O}_{n}(\mathbb{F})\right)$
\item $D$ is additive on $\mathbb{F}$, $D(1) = 0$, and $D/j$, as a map from the group $\mathbb{F}^{\times}$ to $\mathbb{C}$, is a generalized polynomial of degree at most $n$. Here $j$ stands for the identity map defined on $\mathbb{F}$.
\end{enumerate}
\end{thm}

\medskip


Note that according to Lemma \ref{lem_decop_mixed}, if the additive functions
$f_{1}, \ldots, f_{n}, g_{1}, \ldots, g_{n}$ solve equation \eqref{Eq_mixed}, then
the functions $f_{1}, \ldots, f_{n}$ are decomposable functions on the multiplicative group $\mathbb{F}^{\times}$, but this lemma tells nothing about the functions $g_{1}, \ldots, g_{n}$. Now we show that the functions $g_{1}, \ldots, g_{n}$ are also decomposable functions on the multiplicative group $\mathbb{F}^{\times}$.

 \begin{thm}\label{thm_decop_mixed2}
 Let $n$ be a positive integer, $\mathbb{F}\subset \mathbb{C}$ be a field and
 $p_{1}, \ldots, p_{n}, q_{1}, \ldots, q_{n}$ be fixed positive integers fulfilling conditions C(i)--C(iii).
 Assume that the additive functions $f_{1}, \ldots, f_{n}, g_{1}, \ldots, g_{n}\colon \mathbb{F}\to \mathbb{C}$ satisfy functional equation
 \eqref{Eq_mixed}, that is,
 \[
 \sum_{i=1}^{n}f_{i}(x^{p_{i}})g_{i}(x)^{q_{i}}=0
 \]
for each $x\in \mathbb{F}$. Then all the functions $f_{1}, \ldots, f_{n}$ and $g_{1}, \ldots, g_{n}$ are decomposable functions of the group $\mathbb{F}^{\times}$,i.e., all are generalized exponential polynomials.
\end{thm}
\begin{proof}
Due to Lemma \ref{lem_decop_mixed}, the functions $f_{1}, \ldots, f_{n}$ are decomposable functions, hence generalized exponential polynomials on the Abelian group $\mathbb{F}^{\times}$. At the same time, they are assumed to be additive on $\mathbb{F}$. Thus these functions are higher order derivations on the field $\mathbb{F}$. Therefore, $x\longmapsto f_i(x^{p_i})$ is a linear combination of the products of higher order derivations. Let us denote all derivations on $\mathbb{F}$ by $\mathscr{D}$. If $g_i$ are not in $\mathscr{D}$ then there is a summand in the decomposition of $g_i$, which is not in $\mathscr{D}$. Let us denote it by $a$. Then, by Lemma \ref{L_Lin_Ind},  we get that
 \[
 \sum_{i=1}^n f_i(x^{p_i})\cdot(c_i\cdot a(x))^{q_i}=0
 \qquad \left(x\in \mathbb{F}\right),
 \]
with some constants $c_{1}, \ldots, c_{n}\in \mathbb{C}$, since these are exactly those terms that contain $a$. Using the fact that $q_i$ are distinct and  Lemma \ref{L_Lin_Ind}, we have that $a$ and thus $g_{1}, \ldots, g_{n}$ have to be in $\mathscr{D}$. Hence not only $f_{1}, \ldots, f_{n}$, but also  $g_{1}, \ldots, g_{n}$ are decomposable functions.
\end{proof}


\begin{rem}
 Note that the statement of Lemma \ref{lem_decop_mixed} holds true under milder conditions.
 Indeed, compared to equation
 \[
\sum_{i=1}^{n}f_{i}(x^{p_{i}})g_{i}(x^{q_{i}})=0
\qquad
\left(x\in \mathbb{F}\right)
 \]
 that was studied in \cite{GseKis22a},
in equation \eqref{Eq_mixed} the role of the parameters is \emph{not equal}. During the proof it was enough to use only that the parameters $p_{1}, \ldots, p_{n}$ are different. This is important because in such a way it becomes clear that equation
 \begin{equation}\label{eq_Ebanks}
\sum_{i=1}^{n}f_{i}(x^{p_{i}})x^{q_{i}}=0
 \end{equation}
is a special case of equation \eqref{Eq_mixed}.  We remark that equation \eqref{eq_Ebanks} plays a fundamental role in the characterization of higher order derivations, see \cite{Eba15, Eba17, EbaRieSah, GseKisVin18}.
\end{rem}

\section{Main results}\label{SS2.3}

\begin{thm}\label{thm_mixed}
 Let $n$ be a positive integer, $\mathbb{F}\subset \mathbb{C}$ be a field and
 $p_{1}, \ldots, p_{n}, q_{1}, \ldots, q_{n}$ be fixed positive integers fulfilling conditions C(i)--C(iii) and we further assume $q_i<\frac{N}{2}$ for all $i=1, \dots, n$.

\noindent
Suppose that the additive functions $f_{1}, \ldots, f_{n}, g_{1}, \ldots, g_{n}\colon \mathbb{F}\to \mathbb{C}$ satisfy functional equation
\eqref{Eq_mixed}
 for each $x\in \mathbb{F}$. Then there exists a positive integer $l$, there exist exponentials $m_i\colon \mathbb{F}^{\times}\to \mathbb{C}$ and generalized polynomials $P_{i}, Q_{i}\colon \mathbb{F}^{\times}\to \mathbb{C}$ of degree at most $l$ such that
 \begin{equation}\label{eq_mixred}
  f_{i}(x)= P_{i}(x)m_i(x)
  \qquad
  \text{and}
  \qquad
  g_{i}(x)= Q_{i}(x)m_i(x)
  \qquad
  \left(x\in \mathbb{F}^{\times}\right)
 \end{equation}
for each $i=1, \ldots, n$.
\end{thm}

\begin{proof}
Due to Lemma \ref{lem_decop_mixed}, the solutions $f_{1}, \ldots, f_{n}$ and $g_{1}, \ldots, g_{n}$ of equation \eqref{Eq_mixed} are decomposable functions. Hence, they are generalized exponential polynomials on the Abelian group $\mathbb{F}^{\times}$. Accordingly, there exists a positive integer $l$, for all $j=1, \ldots, l, i=1, \ldots, n$ there exist exponentials $m_{j}\colon \mathbb{F}^{\times}\to \mathbb{C}$ and generalized polynomials $P_{i, j}, Q_{i, j}\colon \mathbb{F}^{\times}\to \mathbb{C}$ of degree at most $l$ such that
 \begin{equation}\label{genform}
  f_{i}(x)= \sum_{j=1}^{l} P_{i, j}(x)m_j(x)
  \qquad
  \text{and}
  \qquad
  g_{i}(x)= \sum_{j=1}^{l} Q_{i, j}(x)m_j(x)
  \qquad
  \left(x\in \mathbb{F}^{\times}\right).
 \end{equation}
 We can assume that $\mathbb{F}$ is finitely generated. If so, then the generalized polynomials $P_{i, j}, Q_{i, j}$ are polynomials of degree at most $l$. This technical assumption makes the argument simpler, since whenever we get that $f_i=P_im_i$ and $g_i=Q_im_i$, where $P_i, Q_i$ are polynomials of degree at most $l$ for all finitely generated subfield of $\mathbb{F}$, then  $f_i=P_im_i$ and $g_i=Q_im_i$ holds on $\mathbb{F}$, where $P_i, Q_i$ are generalized polynomial of degree at most $l$.

Suppose contrary that there is an $i\in \{1,\dots, n\}$  such that there are $j_1\ne j_2$ so that $P_{i,j_1}\ne 0$ and $P_{i,j_2}\ne 0$. We can assume that $j_1=1, j_2=2$. Since, by Theorem \ref{T_Poly_Ind}, all $P_i\cdot m_i$ are algebraically independent, the additive functions $\tilde{f}_i$ and $\tilde{g}_i$ are also satisfy equation \eqref{Eq_mixed}, where $$\tilde{f}_i=P_{i, 1}\cdot m_{1}+P_{i, 2}\cdot m_{2} ~~~~ \tilde{g}_i=Q_{i, 1}\cdot m_{1}+Q_{i, 2}\cdot m_{2}.$$

As $P_{i,1},P_{i,2},Q_{i,1},Q_{i,2}$ are now polynomials, they can be written of the form $P(a_1(x), \dots , a_n(x))$, where $P$ is a classical complex polynomial in $n$ variables and $a_1, \dots a_n$ are additive functions from $\mathbb{F}^{\times}$ to $\mathbb{C}$ (usually called logarithmic functions).
Again, by the algebraic independence, equation \eqref{Eq_mixed} holds for each monomial terms of $P(a_1(x), \dots , a_n(x))$. Thus, without loss of generality, we can assume that there are some $\al_{i,1}, \al_{i,2}, \beta_{i,1},\beta_{i,2}\in \C$ such that \begin{equation}\label{genmix}
\bar{f}_i=\al_{i,1}m_1+\al_{i,2}m_2 ~~~~ \bar{g}_i=\beta_{i,1}m_1+\beta_{i,2}m_2,
\end{equation}
which satisfy \eqref{Eq_mixed}.
It is clear that for each monomial terms of polynomials $P_{i,1},P_{i,2},Q_{i,1},Q_{i,2}$ we have the same type of equations as above multiplied by a fix monomial. Furthermore, if we can prove that in this case $\al_{i,1}=\beta_{i,1}=0$ or $\al_{i,2}=\beta_{i,2}=0$ hold for every $i\in\{1,\dots,n\}$, then this argument can be applied for each monomial terms and different indices $j_1\ne j_2$. Hence we get the statement. Therefore, the statement is equivalent to show that
if \eqref{genmix} is a solution of \eqref{Eq_mixed} then $\al_{i,1}=\beta_{i,1}=0$ or $\al_{i,2}=\beta_{i,2}=0$ hold for every $i\in\{1,\dots,n\}$. Thus,
we can write \begin{equation}\label{eqremix2}\sum_{i=1}^n \bar{f}_i(x^{p_i})\bar{g}_i^{q_i}(x)=\sum_{i=1}^n (\al_{i,1}m_1^{p_i}+\al_{i,2}m_2^{p_i})(x)(\beta_{i,1}m_1+\beta_{i,2}m_2)^{q_i}(x)=0.\end{equation}
First we assume that $\beta_{i,1}\ne 0$ for all $i\in \{1, \dots, n\}$. Let $a_{i,1}=\dfrac{\al_{i,1}}{(\beta_{i,1})^{q_i}}$, $a_{i,2}=\dfrac{\al_{i,2}}{(\beta_{i,1})^{q_i}}$ and $b_i=\dfrac{\beta_{i,2}}{\beta_{i,1}}$.
Then the previous equation can be reformulated as
$$\sum_{i=1}^n (a_{i,1}m_1^{p_i}+a_{i,2}m_2^{p_i})(m_1 +b_im_2)^{q_i}=0.$$
In this case all the coefficients of $m_1^{N-l}m_2^l$ have to vanish for each $l=1,\dots, N$. In other words, we have
\begin{equation}\label{eqremix}
    \sum_{i=1}^n\left(\binom{q_i}{l}a_{i,1}b_i^l+\binom{q_i}{N-l}a_{2,i}b_i^{N-l}\right)=0,
    \end{equation}
where $\displaystyle\binom{q_i}{l}$ and $\displaystyle\binom{q_i}{N-l}$, resp. are defined to be $0$ if $l>q_i$, resp. $N-l>q_i$.

Note that till now we did not use the assumption $q_i< \frac{N}{2}$. As $q_i+p_i=N$ and all $p_i,q_i$ are different we have that $p_i>q_i$  and hence we can assume that $p_1>\dots>p_n>q_n>\dots>q_1$. The condition $q_i< \frac{N}{2}$ implies that for every $l\in \{1, \dots, N\}$ and $i\in\{1,\dots,n\}$ at least one of the summand of $\binom{q_i}{l}a_{i,1}b_i^l+\binom{q_i}{N-l}a_{2,i}b_i^{N-l}$ vanishes. Furthermore, for $l=q_n$ the coefficient of $m_1^{N-q_n}m_2^{q_n}$ is $\binom{q_n}{q_n}a_{n,1}\cdot b_n^{q_n}=0$. Similarly, the coefficient $m_1^{q_n}m_2^{N-q_n}$ is $\binom{q_n}{q_n}a_{n,2}\cdot b_n^{N-q_n}=0$. This implies that either $a_{n,1}=a_{n,2}=0$ or $b_n=0$. In the first case we get $\bar{f}_n=0$, then $n$ can be reduced to $n-1$. In the second case we get that $\bar{g}_n=\beta_{n,1}m_1$ and $\bar{g}_n$ has no influence in any term that contains $m_2$. In both cases we can reduce from $n$ to  $n-1$ and now we can proceed by an inductive argument from $n$ to $1$.

Suppose that for $1\le s<n$ we have that $\bar{f}_i\ne 0$ and $\bar{g}_i=\beta_{i,1}m_1$, $\beta_{i,1}\ne 0$ for all $s< i\le n$. Now we consider the coefficients of $m_1^{N-q_s}m_2^{q_s}$  and $m_1^{q_s}m_2^{N-q_s}$ respectively, which are $$ a_{s,1}\cdot b_s^{q_s}=0~~~~ a_{s,2}\cdot b_s^{N-q_s}=0.$$ Indeed, for $i<s$ both $\binom{q_i}{q_s}=\binom{q_i}{N-q_s}=0$, while for $i>s$ the coefficients of $m_1^{q_s}m_2^{N-q_s}$ and $m_1^{N-q_s}m_2^{q_s}$ are 0, since $q_i>q_s$ and $q_i>p_s$. Hence we have that either $a_{s,1}=a_{s,2}=0$ or $b_s=0$. Then either $\bar{f}_s=0$ or $\bar{g}_s=\beta_{s,1}m_1$.

If not all $\beta_{i,1}\ne 0$, then for those $i'\in \{1, \dots, n\}$ such that $\beta_{i',1}=0$ we have $\bar{g}_{i'}=\beta_{i',2}m_2$. In other cases a similar argument works as above. Thus we get that every $\bar{g}_i$ ($i\in \{1, \dots, n\}$) is either $\beta_{i,1}m_1$ or $\beta_{i,2}m_2$. Now we show that if $\bar{g}_i=\beta_{i,1}m_1$ (reps. $\bar{g}_i=\beta_{i,2}m_2$), then $f_i=\al_{i,1}m_1$ (reps. $f_i=\al_{i,2}m_2)$. In this case there are disjoint subsets $I_1, I_2$ of $\{1,\dots, n\}$ such that $I_1\cap I_2=\{1,\dots, n\}$ and equation \eqref{eqremix2} can be written $$\sum_{i\in I_1} (\al_{i,1}m_1^{p_i}+\al_{i,2}m_2^{p_i})(x)(\beta_{i,1}m_1)^{q_i}(x) + \sum_{i\in I_2} (\al_{i,1}m_1^{p_i}+\al_{i,2}m_2^{p_i})(x)(\beta_{i,2}m_2)^{q_i}(x)=0. $$
Let $c_1$ and $c_2$ denote the coefficients of the monomial $m_1^{p_i}m_2^{q_i}$ and  $m_1^{q_i}m_2^{p_i}$ respectively.  Then $$c_1=\begin{cases}0& \textrm{ if }i\in I_1 \\
\al_{i,2}\beta_{i,1}^{q_i} & \textrm{ if }i\in I_2
\end{cases} ~~~~~~~~~ c_2= \begin{cases}
\al_{i,1}\beta_{i,2}^{q_i} & \textrm{ if }i\in I_1\\
0& \textrm{ if }i\in I_2
\end{cases}.$$
Eliminating those terms where $\bar{g}_i\equiv 0$, we get that if $\beta_{i,1}\ne 0$, then $\al_{i,2}=0$ and similarly if $\beta_{i,2}\ne 0$, then $\al_{i,1}=0$. This completes the proof of the theorem.
\end{proof}





As we noticed, the statement of Theorem \ref{thm_mixed} is not necessarily true without the assumption $q_i\le \frac{N}{2}$ for all $i\in \{1, \dots, n\}$. In the following we show an example where neither $f_i$, nor $g_i$ is of the form $P\cdot m$.

\begin{ex}\label{Exmix}
Let $n=2$ and $p_1=2, p_2=1$ in equation \eqref{Eq_mixed}, i.e., assume that we have
\[
f_{1}(x^{2})g_{1}^{N-2}(x)+f_{2}(x)g_{2}^{N-1}(x)=0
\qquad
\left(x\in \mathbb{F}\right).
\]
Let further $m_{1}$ and $m_{2}$ be different exponentials on $\mathbb{F}^{\times}$ and define the functions
$f_{1}, f_{2}, g_{1}, g_{2}$ on $\mathbb{F}^{\times}$ by
\[
\begin{array}{rcl}
f_{1}(x)&=& m_{1}(x)-m_{2}(x)\\
f_{2}(x)&=& m_{2}(x)-m_{1}(x) \\
g_{1}(x)&=& m_{1}(x)+m_{2}(x)\\
g_{2}(x)&=& m_{1}(x)+m_{2}(x)
\end{array}
\qquad
\left(x\in \mathbb{F}^{\times}\right).
\]
Then
\begin{multline*}
    f_{1}(x^{2})g_{1}^{N-2}(x)+f_{2}(x)g_{2}^{N-1}(x)
    \\
    =
    \left(m_{1}^{2}(x)-m^{2}_{2}(x)\right)\cdot \left(m_{1}(x)+m_{2}(x)\right)^{N-2}+
    \left(m_{2}(x)-m_{1}(x)\right)\cdot \left(m_{1}(x)+m_{2}(x)\right)^{N-1}
    \\
    =
    \left(m_{1}(x)-m_{2}(x)\right)\cdot (m_{1}(x)+m_{2}(x))\cdot \left(m_{1}(x)+m_{2}(x)\right)^{N-2}+
    (-1)\cdot \left(m_{1}(x)-m_{2}(x)\right)\cdot \left(m_{1}(x)+m_{2}(x)\right)^{N-1}=0
\end{multline*}
for all $x\in \mathbb{F}^{\times}$. Since the involved functions $f_{1}, f_{2}, g_{1}, g_{2}$ were assumed to be additive, as well, we get that if $\varphi_{1}, \varphi_{2}\colon \mathbb{F}\to \mathbb{C}$ are homomorphisms and we consider the following functions
\[
\begin{array}{rcl}
f_{1}(x)&=& \varphi_{1}(x)-\varphi_{2}(x)\\
f_{2}(x)&=& \varphi_{2}(x)-\varphi_{1}(x) \\
g_{1}(x)&=& \varphi_{1}(x)+\varphi_{2}(x)\\
g_{2}(x)&=& \varphi_{1}(x)+\varphi_{2}(x)
\end{array}
\qquad
\left(x\in \mathbb{F}^{\times}\right),
\]
then the above equation is fulfilled for all $x\in \mathbb{F}$. This shows that the condition $q_{i}< \dfrac{N}{2}$ cannot be  omitted from Theorem \ref{thm_mixed} in general.
\end{ex}

\begin{rem}
It is very important to emphasize that when we talk about the solutions of equation \eqref{Eq_mixed}, we look for the solutions among \emph{additive} functions.
If we omit the condition of additivity but the solutions are still exponential polynomials of degree different from zero, then with a similar argument as above, one can show that there could be found solutions having the similar form (i.e., there are at least two different exponentials in the solutions).
To see this, let $m_{1}, m_{2}$ be different exponentials and $a$ be an additive function on the multiplicative group $\mathbb{F}^{\times}$ and consider the functions
\[
\begin{array}{rcr}
f_{1}(x)&=&a(x)(m_{1}(x)-m_{2}(x))\\
f_{2}(x)&=&-2a(x)(m_{1}(x)-m_{2}(x))\\
g_{1}(x)&=& m_{1}(x)+m_{2}(x)\\
g_{2}(x)&=& m_{1}(x)+m_{2}(x)
\end{array}
\qquad
\left(x\in \mathbb{F}^{\times}\right).
\]
An easy computation shows that in this case we have
\[
f_{1}(x^{2})g_{1}^{N-2}(x)+f_{2}(x)g_{2}^{N-1}(x)=0
\]
holds for all $x\in \mathbb{F}$.
It is important to emphasize however that these functions will be additive only if the functions $a$ and $m_{1}, m_{2}$ appearing in the above representations satisfy $m_1(x)-m_2(x)=0$ if $a(x)\neq 0$, i.e., the previous system of equation becomes trivial.
From Lemma \ref{lem2tag2} one can also deduce that the above functions are not additive in general.
\end{rem}


As an intermediate result in connection to Theorem \ref{thm_mixed} and Example \ref{Exmix} is the following example we show that the assumptions of Theorem \ref{thm_mixed} are not sharp, as not all of the parameters should satisfy $q_i<\frac{N}{2}$. As a counterpart of Example \ref{Exmix}, we prove that the solutions in the following case are of the form $f_i=P_im$ and $g_i=Q_im$, for generalized polynomials $P_i, Q_i$ and exponential $m$.

\begin{ex}\label{Exmix2}
Consider equation
\begin{equation}\label{Ex2.1}f_1(x^k)g_1(x)^{N-k}+f_2(x^{N-l})g_2(x)^l=0\end{equation}
for all $x\in \mathbb{F}$,
where $l+1<k\le \frac{N}{2}$.

As we showed above in this case there are solutions that can be represented as
\[
f_{i}(x)= a_{i, 1}m_{1}(x)+a_{i, 2}m_{2}(x)
\qquad
g_{i}(x)= m_{1}(x)+b_{i}m_{2}(x)
\qquad
\left(x\in \mathbb{F}\right).
\]
Now we show that if none of $f_i$ and $g_i$ vanishes, then $b_1=b_2=0$.
Calculating the coefficients of the term $m_1^sm_2^{N-s}$, we can observe that it is taken only from the first term $f_1(x^k)g_1(x^{N-k})$, if $l< s< N-l$.
If further $s<k$, then the coefficient of $m_1^sm_2^{N-s}$ satisfies
$$\binom{N-k}{s}a_{1,2}b_1^{N-k-s}=0.$$ Similarly, for the coefficient of $m_1^{N-s}m_2^s$ we get that $$\binom{N-k}{s}a_{1,1}b_1^{s}=0.$$
These equations imply that either $b_1=0$ or $a_{1,1}=a_{1,2}=0$. The latter is not possible as $f_1$ is not identically zero, thus $b_1=0$. Hence equation \eqref{Ex2.1} reduces to
$$\big(a_{1, 1}(m_{1}^k(x)+a_{1, 2}m_{2}^k(x)\big)m_1^{n-k}(x)+ \big(a_{2, 1}m_{1}^{N-l}(x)+a_{2, 2}m_{2}^{N-l}(x)\big)\big(m_{1}(x)+b_2m_{2}(x)\big)^{l}=0.$$
In this case  $a_{2,2}b_2^N=0$ can be obtained as the coefficient of $m_2^N$, and $a_{2,1} b_2^{l}=0$ is given as the coefficient $m_1^{N-l}m_2^{l}$, if $k\ne l$. 
A similar argument as above shows that $b_2=0$. Hence we can assume that $g_1=g_2=m_1$.
It is straightforward to verify that $f_1=a_{1,1}m_1$ and $f_2=a_{1,2}m_1$ in this case, which as in the proof of Theorem \ref{thm_mixed} implies that every solution of \eqref{Ex2.1} is of the
form $f_i=P_im$ and $g_i=Q_im$, where $P_i,Q_i$ are generalized polynomials and $m$ is an exponential on $\mathbb{F}^{\times}$.

Note that if we omit the assumption that there exists an $s$ such that $l<s<k$, then the coefficient $m_1^sm_2^{N-s}$ for $l<s<N-l$ appears in more then one term in the expansion of $f_1(x^k) (g_1(x))^{N-k}$, which makes the whole calculation much more complicated and it is not clear whether we can get similar result.
\end{ex}

\begin{rem}\label{rem_vegy_pol}Although Example \ref{Exmix} shows that equation \eqref{Eq_mixed} cannot automatically be reduced to solutions of type \eqref{eq_mixred}, by Theorem \ref{T_Poly_Ind}, algebraic independence guarantees that if a system of solutions is of the form \eqref{genform}, then there are also solutions of the form \eqref{eq_mixred} just keeping the terms containing a given $m$ in each $f_i$ and $g_i$. These reduced functions are additive as well and satisfy \eqref{Eq_mixed}. Therefore, from now on we are dealing with those solutions that are of the form $f_i=P_im$ and $g_i=Q_im$. By the equivalence relation $\sim$ introduced in Remark \ref{rem2.4}, we can assume that $f_i(x)=P_i(x)\cdot x$ and $g_i(x)=Q_i(x)\cdot x$, where $P_i, Q_i$ are generalized polynomial on $\mathbb{F}^{\times}$ of degree at most $K$. Hence by Theorem \ref{thm_KisLac}, these are derivations of order at most $K$ on any finitely generated subfield.
Thus, we may restrict ourselves to functions that are of the form
\begin{equation}\label{eq_mix_alap}
 f_{i}(x)= P_{i}(x)\cdot x =D_i(x)
 \qquad
\text{and}
\qquad
g_{i}(x)= Q_{i}(x)\cdot x=\widetilde{D}_i(x)
\qquad
\left(x\in \mathbb{F}^{\times}, i=1, \ldots, n\right).
 \end{equation}
\end{rem}

\medskip

Every higher order derivation on $\mathbb{F}$ is a differential operator on any finitely generated subfield of $\mathbb{F}$ (see Theorem \ref{thm_KisLac} and \cite{KisLac18}). Hence on these fields the solutions are differential operators. Moreover, if every solution on any finitely generated subfield of $\mathbb{F}$ is a differential operator of order at most $n$, then every solution on  $\mathbb{F}$ is a derivation of order at most $n$. By this fact, from now on, instead of finding solutions as higher order derivations we may restrict ourselves to look for differential operators as solutions.

The space of differential operators is a linear space. On the other hand, settling a useful basis is not trivial. The following lemma provides such a basis. Its proof is based on generalized moment sequences and the notion of (multivariate) Bell polynomials. For further details we refer to \cite[Subsection 3.4]{GseKis22a}.

\begin{lem}\label{L:dermom}
 Let $\mathbb{F}\subset \mathbb{C}$ be a field, $r$ be a positive integer and
 $d_{1}, \ldots, d_{r}\colon \mathbb{F}\to \mathbb{F}$ be linearly independent derivations. For all multi-index $\alpha\in \mathbb{N}^{r}$, $\alpha= \left(\alpha_{1}, \ldots, \alpha_{r}\right)$ define the function $d^{\alpha}(x)\colon \mathbb{F}\to \mathbb{C}$ by

\[
 d^{\alpha}(x)=
  d_{1}^{\alpha_{1}} \circ \cdots \circ d_{r}^{\alpha_{r}}(x
=
\underbrace{d_{1}\circ \cdots \circ d_{1}}_\text{$\alpha_{1}$ times}
\circ \cdots \circ \underbrace{d_{r}\circ \cdots \circ d_{r}}_{\text{$\alpha_{r}$ times}}(x) \qquad
\qquad
\left(x\in \mathbb{F}^{\times}\right).
\]

Then $(d^{\alpha}(x))_{\alpha \in \mathbb{N}^{r}}$ constitute a basis of the differential operators constructed by $d_1, \dots, d_r$ in $\mathbb{F}$.
\end{lem}

For a multi-index $\alpha= \left(\alpha_{1}, \ldots, \alpha_{r}\right)$ we denote $|\alpha|=\sum_{i=1}^n \alpha_k$.
\begin{cor} \label{cor:alg_ind}

 By Theorem \ref{T_Poly_Ind}, the elements of $d^{\alpha}$ are algebraically independent, for all $\alpha\in \cup_{r\in \mathbb{N}}\mathbb{N}^r$. Let $d_1, \dots d_r$ be derivations as in Lemma \ref{L:dermom} and $\alpha_1, \dots, \alpha_n \in  \cup_{r\in \mathbb{N}}\mathbb{N}^r$.  Then equation \eqref{Eq_mixed} can be written in the
   following form
    \[
    \sum_{i=1}^n f_i(x^{p_i})(g_i(x))^{q_i}=\sum_{i=1}^n \big(\sum_{|\alpha|<k_i}d^{\alpha}(x^{p_i})\big)\big(\sum_{|\beta|<l_i}d^{\beta}(x)\big)^{q_i}=0
    \qquad
     \left(x\in \mathbb{F}\right).
    \]
    Now we fix an $\alpha\in \mathbb{N}^r$ such that $|\alpha|$ is maximal in $f_i$'s and $\beta\in \mathbb{N}^t$ is taken to be maximal in those $g_i$ where $d^{\alpha}$ appears as a summand in $f_i$. Then by algebraic independence  we can restrict to those $\alpha'\in \mathbb{N}^r$ and $\beta'\in \mathbb{N}^t$ such that $\alpha'_k\le \alpha_k$ ($k=1,\dots, r$) and $\beta'_j\le \beta_j$ ($j=1,\dots, t$). This we denote by $\alpha'\le \alpha$ and $\beta'\le \beta$, respectively. Hence
    \[
    \sum_{i=1}^n f_i(x^{p_i})(g_i(x))^{q_i}=\sum_{i=1}^n \big(\sum_{k\le |\alpha|}\hat{d}^{k}(x^{p_i})\big)\big(\sum_{l\le|\beta|}\hat{d}^{l}(x)\big)^{q_i}=0
    \qquad
     \left(x\in \mathbb{F}\right),
    \]
     where $\hat{d}$ is an arbitrary derivation (of order 1).
    In other words, we can substitute $d_1, \dots, d_r$ by $\hat{d}$ in $d^{\alpha'},  d^{\beta'}$ in each case whenever $\alpha'<\alpha$ and $\beta'<\beta$.
\end{cor}

Our next aim is to understand the arithmetic of composition of derivations of the form $\underbrace{d\circ \cdots \circ d}_\text{$k$ times}(x)$, where $d$ is a derivation (of order 1), $k\in \mathbb{N}$, as they are building blocks of differential operators.
Lemma \ref{L:dermom}, together with \cite[Proposition 1]{FecGseSze21}, implies the following statement.  

\begin{prop}\label{prop1}
 Let $\mathbb{F}\subset \mathbb{C}$ be a field and $d\colon \mathbb{F}\to \mathbb{C}$ a derivation. For all positive integer $k$ we define the function $d^{k}$ on
 $\mathbb{F}$ by
 \[
  d^{k}(x)=
  \underbrace{d\circ \cdots \circ d}_\text{$k$ times}(x)
  \qquad
  \left(x\in \mathbb{F}\right).
 \]
Then for all positive integer $p$ we have
\[
d^{k}(x^{p})=
 \sum_{\substack{l_{1}, \ldots, l_{p}\geq 0\\ l_{1}+\cdots+l_{p}=k}}
 \binom{k}{l_{1}, \ldots, l_{p}}\cdot d^{l_{1}}(x)\cdots d^{l_{p}}(x)
 \qquad
 \left(x_{1}, \ldots, x_{p}\in \mathbb{F}\right),
\]
where the conventions $d^{0}= \mathrm{id}$ and $\displaystyle\binom{k}{l_{1}, \ldots, l_{p}}= \dfrac{k!}{l_{1}! \cdots l_{p}!}$ are  adopted. 

Reordering the previous expression we can get the following
\[
 d^{k}(x^{p})= \sum_{\substack{ j_{1}+\cdots+j_{s}=p'< p\\
 j_1+2j_2+\dots+sj_s=k}}  \binom{k}{\underbrace{1,\dots, 1}_{j_{1}}, \ldots, \underbrace{s,\dots, s}_{j_{s}}}\cdot  \prod_{t=1}^s \frac{1}{(j_t!)}\cdot  \binom{p}{\underbrace{1,\dots,1}_{p'}}\cdot (d(x))^{j_1}\cdots (d^{s}(x))^{j_s}\cdot x^{p-p'},
 \qquad
 \left(x\in \mathbb{F}\right)
 \]

where $j_1,\dots, j_s$ denotes the number of $d(x),\dots, d^s(x)$ in a given composition of $d^k(x^p)$.
\end{prop}

\begin{thm}\label{T_mix_fo1}
 Let $n$ be a positive integer, $\mathbb{F}\subset \mathbb{C}$ be a field and
 $p_{1}, \ldots, p_{n}, q_{1}, \ldots, q_{n}$ be fixed positive integers fulfilling conditions C(i)--C(iii).
 Assume that the additive functions $f_{1}, \ldots, f_{n}, g_{1}, \ldots, g_{n}\colon \mathbb{F}\to \mathbb{C}$ defined by
 \begin{equation}\label{eq_mixred2}
  f_{i}(x)= D_i(x)
  \qquad
  \text{and}
  \qquad
  g_{i}(x)= \widetilde{D}_i(x)
  \qquad
  \left(x\in \mathbb{F}^{\times}\right)
 \end{equation}
for each $i=1, \ldots, n$, satisfy functional equation \eqref{Eq_mixed} on
$\mathbb{F}$, where for all index $i=1, \ldots, n$, the mappings $D_i, \widetilde{D}_i$ are higher order derivations on $\mathbb{F}$.
Then one of the following two alternatives hold:
\begin{enumerate}
    \item[(A)] there exists $i\in \{1, \dots, n\}$ such that $g_i(x)=c_i\cdot x$ and $f_i$ contains, as a summand, a derivation of order $K$, where $K$ is the maximum order in each $D_i, \widetilde{D}_i$.
    \item[(B)] $f_i$ and $g_i$ is of order at most 1 for all $i\in \{1, \dots, n\}$ and there are $i_1, i_2\in \{1, \dots, n\}, i_1\ne i_2$ such that $q_{i_2}=q_{i_1}+1$ ($p_{i_2}=p_{i_1}-1$) and $f_{i_1},g_{i_1}, f_{i_2}, g_{i_2}$ is of the form
    \begin{equation}\label{eq1rendfg}
    \begin{split}
    f_{i_1}=\la_{i_1,1} d(x)+\la_{i_1,0}x,& ~ ~ g_{i_1}(x)=\tla_{i_1,1} d(x)+\tla_{i_1,0}x,\\
    f_{i_2}(x)=\la_{i_2,0}x,& ~ ~ g_{i_2}(x)=\tla_{i_2,1} d(x)+\tla_{i_2,0}x,
    \end{split}
    \end{equation}
    where $\la_{i_1,1}, \tla_{i_1,1}, \la_{i_2,0}, \tla_{i_2,1}$ are nonzero complex numbers satisfying
    \begin{equation}\label{eq1rendeq}
    p_{i_1}\cdot\la_{i_1,1}\cdot(\tla_{i_1,1})^{q_{i_1}}+\la_{i_2,0}\cdot (\tla_{i_2,1})^{q_1+1}=0.
    \end{equation}
\end{enumerate}
\end{thm}

\begin{proof}

 Substituting the form \eqref{eq_mix_alap} to equation \eqref{Eq_mixed}, we arrive to
 \[
 0= \sum_{i=1}^{n}f_{i}(x^{p_{i}})g_{i}^{q_{i}}(x)
 =
 \sum_{i=1}^{n}P_{i}(x^{p_{i}})x^{p_{i}}\cdot Q_{i}^{q_{i}}(x)x^{q_{i}}=
 \sum_{i=1}^{n} D_{i}(x^{p_{i}})(x)(\widetilde{D}_{i}(x))^{q_{i}}
 \qquad
 \left(x\in \mathbb{F}^{\times}\right).
 \]

 For simplicity, from now on we assume that $\mathbb{F}$ is finitely generated. Then the corresponding functions can be represented as
 \[
 f_i(x)=D_{i}(x)= \sum_{|\alpha|<k_i}d^{\alpha}(x)
 \qquad
 \text{and}
 \qquad
 g_i(x)=\widetilde{D}_{i}(x)= \sum_{|\beta|<l_i}d^{\beta}(x)
 \qquad
\left(x\in \mathbb{F}^{\times}\right),
\]
where $\alpha$ and $\beta$ are running multi-indices and $k_i$ and $l_i$ are some natural numbers depending on $f_i$ and $g_i$, respectively. By Corollary \ref{cor:alg_ind}, we can take only derivation $d$ in each composition of each $d^{\alpha}$ and $d^{\beta}$ so that the corresponding $f_i$'s and $g_i$'s still satisfy \eqref{eq_mix_alap}. We can represent these functions as
 \[
 f_{i}(x)=D_i(x)= \sum_{j=0}^{k_i}\lambda_{i, j}d^{j}(x)
 \qquad
 \text{and}
 \qquad
 g_{i}(x)=\widetilde{D}_i(x)= \sum_{j=0}^{l_i}\widetilde{\lambda}_{i, j}d^{j}(x)
 \qquad
 \left(x\in \mathbb{F}^{\times}\right)
\]
with an appropriate derivation $d\colon \mathbb{F}\to \mathbb{C}$ and complex constants $\lambda_{i, j}$ ($i=1, \ldots, n, j=0, \ldots, k_i$) $\widetilde{\lambda}_{i, j}$ ($i=1, \ldots, n, j=0, \ldots, l_i$), where $k_i$ and $l_i$ denote the the highest order term of those derivations that appear in $f_i$ and $g_i$, respectively. This means that $\lambda_{i, k_i}, \widetilde{\lambda}_{i, l_i}$ are nonzero for all $i=1,\dots,n$.

If we write these representations into \eqref{Eq_mixed}, we especially get that
\[
 \sum_{i=1}^{n}\left(\sum_{j=0}^{k_i}\lambda_{i, j}d^{j}(x^{p_{i}}) \right)
 \cdot \left(\sum_{j=0}^{l_i}\widetilde{\lambda}_{i, j}d^{j}(x)\right)^{q_{i}}=0
 \qquad
  \left(x\in \mathbb{F}^{\times}\right).
\]
By introducing the following quantities
\[
 S(p_{i}, k_i-1)= \sum_{j=0}^{k_i-1}\lambda_{i, j}d^{j}(x^{p_{i}})
 \qquad
 \text{and}
 \qquad
 T(q_{i}, l_i-1)= \left(\sum_{j=0}^{l_i}\widetilde{\lambda}_{i, j}d^{j}(x)\right)^{q_{i}}- (\widetilde{\lambda}_{i, l_i}d^{l_i}(x))^{q_i}
 \qquad
 \left(x\in \mathbb{F}^{\times}\right).
\]

Dividing the above sum to smaller ones, we get
\begin{multline*}
 \sum_{i=1}^{n}\left(\sum_{j=0}^{k_i}\lambda_{i, j}d^{j}(x^{p_{i}}) \right)
 \cdot \left(\sum_{j=0}^{l_i}\widetilde{\lambda}_{i, j}d^{j}(x)\right)^{q_{i}}
 \\
 =
 \sum_{i=1}^{n}\left(\lambda_{i, k_i}d^{k_i}(x^{p_{i}})+S(p_i, k_i-1)  \right)
 \cdot \left( (\lambda_{i, l_i}d^{l_i}(x))^{q_i}+T(q_i,l_i-1)\right)
 \qquad
\left(x\in \mathbb{F}^{\times}\right).
\end{multline*}
Let $K=\max_i \left\{k_i+l_i\cdot q_i\right\}$ and for simplicity let us assume that this $K$ realized for indices $i\in\{1,\dots,m\}$ for some $m\le n$. Suppose that $k_1$ maximal. Now we assume that $l_1\ne 0$. Otherwise we immediately get the result, that the order of $f_1$ is maximal and $g_1=c_1\cdot x$.    

Now, we investigate the coefficient of $d^{k_1}(x)(d^{l_1}(x))^{q_1}$ in the expansion. As $K$ is maximal this can be taken only from the product of the first terms of the previous expression, i.e., from
\begin{equation}\label{eq_mixer}\sum_{i=1}^m \lambda_{i, k_i}(\widetilde{\lambda}_{i, l_i})^{q_i}d^{k_i}(x^{p_{i}})(d^{l_i}(x))^{q_{i}}.
\end{equation}

By Proposition  \ref{prop1}, we have that
\begin{multline*}
d^{k_i}(x^{p_{i}})= \sum_{\substack{ j_{1}+\cdots+j_{s}=p'< p_i\\
 j_1+2j_2+\dots+sj_s=k_i}}  \binom{k_i}{\underbrace{1,\dots, 1}_{j_{1}}, \ldots, \underbrace{s,\dots, s}_{j_{s}}}\cdot  \prod_{t=1}^s \frac{1}{(j_t!)}\cdot  \binom{p_i}{\underbrace{1,\dots,1}_{p'}}\cdot (d(x))^{j_1}\cdots (d^{s}(x))^{j_s}\cdot x^{p_i-p'}
  \\
  =\begin{cases} \left(p_i\cdot d^{k_i}(x)\cdot x^{p_i-1}+ R(p_i,k_i)\right),& \text{ if } k_i\ge 1\\
  x^{p_i}, & \text{ if } k_i=0
  \end{cases}   \qquad
\left(x\in G, i=1, \ldots, n\right),
\end{multline*}
where each term in $R(p_i, k_i)$ contains the product of at least two derivations of order less than $k_i$.

The term $d^{k_1}(x)(d^{l_1}(x))^{q_1}$ is automatically appears in the expansion of $d^{k_1}(x^{p_i})(d^{l_1}(x))^{q_1}$. As it have to vanish in the sum we should observe how we can get terms of the form  $d^{k_1}(x)(d^{l_1}(x))^{q_1}$ from the expansion of $d^{k_i}(x^{p_i})(d^{l_i}(x))^{q_i}$ for some $i\in \{2, \dots, m\}$.

We have to distinguish several cases and subcases.
\begin{enumerate}[{Case} 1.]
\item  $d^{k_1}(x)$ stems from the expansion of $d^{k_i}(x^{p_i})$. Then, by the maximality of $k_1$, we obtain $k_i=k_1$. As $k_1+l_1q_1=k_i+l_iq_i$ and $q_i\ne q_1$ we have that $l_1\ne l_i$ and hence we cannot get $(d^{l_1}(x))^{q_1}$ from $(d^{l_i}(x))^{q_i}$, which is a contradiction.

\item  $d^{k_1}(x)$ stems from $(d^{l_i}(x))^{q_i}$. Then $k_1=l_i$ and we have several subcases.

\smallskip
\noindent
Case 2.1. If $q_i>1$, then $l_1=l_i$, and $d^{k_1}(x)(d^{l_1}(x))^{q_1}$ can be reformulated as $d^{k_1}(x)(d^{k_1}(x))^{q_1}$. Similarly, $d^{k_i}(x^{p_i})(d^{l_i}(x))^{q_i}$ can be reformulated as $d^{k_i}(x^{p_i})(d^{k_1}(x))^{q_i}$.

\smallskip
\noindent
Case 2.1.1. If $k_1=k_i$, then we get a contradiction as in Case 1.

\smallskip
\noindent
Case 2.1.2. If $k_1\ne k_i$, i.e., $k_1>k_i$, then  $d^{k_1}(x)$ can only stems from $(d^{k_1}(x))^{q_i}$. As $q_i>1$, it follows that $l_1=k_1$, which implies $k_i=0$. Thus  $d^{k_1}(x)(d^{k_1}(x))^{q_1}$ have to be the same as $(d^{k_1}(x))^{q_i}$. Hence $q_i=q_1+1$, and hence $p_1=p_i+1>2$.

\smallskip
\noindent
Case 2.1.2.1. $k_1>1$. Since $p_1>2$, the expansion of
$d^{k_1}(x^{p_1})(d^{k_1}(x))^{q_1}$ contains a term of the form
$d(x)d^{k_1-1}(x)(d^{k_1}(x))^{q_1}$. This
cannot appear in the expansion of $d^{k_j}(x^{p_j})(d^{l_j}(x))^{q_j}$ for any $j\in \{2, \dots, m\}$. Indeed, if $k_1=k_j$ and $d(x)d^{k_1-1}(x)$ stems from $d^{k_j}(x^{p_j})$, then we get a contradiction as in Case 1. If $k_1=k_j$ and $d(x)d^{k_1-1}(x)$ stems from $(d^{l_j}(x))^{q_j}$, then $k_1=2, l_j=1, q_j=2$, and $(d^{k_1}(x))^{q_1}$ stems from the expansion of $d^{k_j}(x^{p_j})$. Hence by the maximality of $k_1$, we get $k_j=k_1$ and $q_1=1$. On the other hand, if $q_1=1$, then $q_i=2$ as well as $q_j$, hence $q_i=q_j$, but  $k_i=0$ and $k_j=k_1$, which is a contradiction.

\smallskip
\noindent
Case 2.1.2.2. $k_1=1$.
As $k_i<k_1$, then $k_i=0$. In this case $k_1=l_1=l_i=1$ and $k_i=0$. There is no other $k_j=1$ for any $j\in \{2, \dots m\}\setminus \{i\}$, since $l_1$ is maximal, and if $l_1=0$, then $k_1$ is not maximal.
If $k_j=0$, then the corresponding term is $x^{p_j}(d^{l_j})^{q_j}$, where $l_j\cdot q_j=q_1+1=K$ could be possible, but by the maximality of $k_1$ and $l_1$ the term $(d^{l_j})^{q_j}$ cannot be eliminated. Thus there is no other $k_j=0$ for any $j\in \{2, \dots m\}\setminus \{i\}$.
Therefore,  $i=m=2$ and the term $d(x)(d(x))^{q_1}$ can only be eliminated using the terms corresponding to $k_1$ and $k_2$. Namely,
$$f_1=\la_{1,1}\cdot d(x)+\la_{1,0}x, ~ ~ g_1(x)=\tla_{1,1}\cdot d(x)+\tla_{1,0}x$$
$$f_2(x)=\la_{2,0}x, ~ ~ g_2(x)=\tla_{2,1}\cdot d(x)+\tla_{2,0}x$$
and $q_2=q_1+1$, so that
$$\la_{1,1}(\tla_{1,1})^{q_1}p_1\cdot x^{p_1-1} d(x)^{q_1+1}+\la_{2,0}(\tla_{2,1})^{q_1+1}\cdot x^{p_1-1} d(x)^{q_1+1}=0,$$ hence $p_1\la_{1,1}(\tla_{1,2})^{q_1}+\la_{2,0} (\tla_{2,1})^{q_1+1}=0.$
Repeating the whole argument recursively for $K_j=\max(k_i+l_iq_i)\setminus \{K_1, \dots, K_{j-1}\}$ we get that either there is an $f_i$ of degree at least 2 and then there exists a $f_i$ of maximal degree $K_j$ such that $g_i=c_i\cdot x$, or every $f_i$ has degree at most $1$, and the corresponding functions has a strong connection as described above.

\smallskip
\noindent
Case 2.2. If $q_i=1$, then we have that $d^{k_1}(x)$ stems from $d^{l_i}(x)$, so $k_1=l_i$ and $l_1\cdot q_1=k_i$. In this case $d^{k_1}(x)(d^{l_1}(x))^{q_1}$ stems from the expansion of $d^{k_1}(x^{p_1})(d^{l_1}(x))^{q_1}$ and of $d^{k_i}(x^{p_i}) (d^{l_i})^{q_i}=d^{l_1q_1}(x^{N-1}) d^{k_1}(x)$.
As $q_i=1$, we have that $q_1>1$ and $p_i=N-1>1$, as $p_i, q_i$ are distinct. Furthermore, as $k_1$ is  maximal and $l_1\ne 0$ (otherwise we automatically get the result), we obtain $k_1\ge k_i\ge l_1q_1 \ge 2$.

Now we take the term $d^{l_1q_1-1}(x)d(x) d^{k_1}(x)$, this term stems from $d^{k_i}(x^{p_i}) (d^{l_i})^{q_i}=d^{l_1q_1}(x^{N-1}) d^{k_1}(x)$  as $p_i=N-1>1$.
Therefore, there must be a $j\in \{1, \dots, m\}$ with $j\ne i$ so that $d^{l_1q_1-1}(x)d(x) d^{k_1}(x)$ stems from
$d^{k_j}(x^{p_j}) (d^{l_j})^{q_j}$. As we have $d^{l_1q_1-1}(x)d(x) d^{k_1}(x)$ is the product of 3 terms and $q_j\ne 1$ as $q_i=1$ we have get that $q_j=2$.
We have three cases.

\smallskip
\noindent
Case 2.2.1.  $d(x) d^{k_1}(x)=(d^{l_j})^2$. Then $k_1=1$, which contradicts the fact that $k_1\ge 2$.

\smallskip
\noindent
Case 2.2.2.  $d^{l_1q_1-1}(x) d^{k_1}(x)=(d^{l_j})^2$. Then $k_1=l_1q_1-1$,  then $k_i=l_1q_1>k_1$, which contradicts the maximality of $k_1$.

\smallskip
\noindent
Case 2.2.3.  $d(x) d^{l_1q_1-1}(x)=(d^{l_j})^2$. Then $l_1q_1=2$, and as $q_1\ge 2$ and $l_1\ge 1$ we get that $l_1=1$ and $q_1=2$. This implies $q_j=q_1$ and hence $j=1$. Thus, $d^{k_i}(x^{p_i}) (d^{l_i})^{q_i}=d^{2}(x^{N-1}) d^{k_1}(x)$. Now we take the term $d^{2}(x) d^{k_1}(x)$. As it clearly stems from the expansion of the previous expression, there should be a $j\in\{2,\dots, m\}$, where $j\ne i$ (and also $j\ne 1$) so that $d^{2}(x) d^{k_1}(x)$ stems from the term $d^{k_j}(x^{p_j})(d^{l_j}(x))^{q_j}$. Then $q_j=1$, which is a contradiction as $q_i=1$ and $i\ne j$.
\end{enumerate}

\medskip
Summarizing these results we either have $l_1=0$, which indicates that the order of $f_1$ is
maximal and $g_1 (x) = c_1 \cdot x \, (x\in \mathbb{F})$ for some complex number $c_1$. Otherwise, for every $i=1, \ldots, n$ the functions $f_i$ and $g_i$ are of degree at most $1$.  Furthermore, if $f_{i_1}$ and $g_{i_1}$ is  of the form
$$ f_{i_1}=\la_{i_1,1} d(x)+\la_{i_1,0}x, ~ ~ g_{i_1}(x)=\tla_{i_1,1} d(x)+\tla_{i_1,0}x$$ then there exists an $i_2\in \{1, \dots, n\}$ such that
$$g_{i_1}=\la_{i_2,0}x, ~ ~ g_{i_2}(x)=\tla_{i_2,1} d(x)+\tla_{i_2,0}x$$ where
and $q_{i_2}=q_{i_1}+1$ ($p_{i_2}=p_{i_1-1}$), so that $\la_{i_1,1}, \tla_{i_1,1}, \la_{i_2,0}, \tla_{i_2,1}$ are nonzero complex numbers satisfying
    \begin{equation*}
    p_{i_1}\cdot\la_{i_1,1}\cdot(\tla_{i_1,1})^{q_{i_1}}+\la_{i_2,0}\cdot (\tla_{i_2,1})^{q_1+1}=0.
    \end{equation*}
\end{proof}
\begin{rem}
We get more than it is stated in Theorem \ref{T_mix_fo1}. Namely, if alternative (A) happens then the maximal order of $K$ is greater or equal to $k_i+q_i\cdot l_i$ for all $i\in \{1, \dots, n\}$, where $k_i$ is the maximal order of $f_i$ and $l_i$ is the maximal order of $g_i$. Thus, if $k_j+q_j\cdot l_j=K$ and $k_j\ne K$ for some $j\in \{1,\dots, n\}$, then we typically $k_i$ and $l_i$ is much smaller than $K$.
On the other hand, it is worth mentioning that the possibility that such $k_i, l_i$ do exist cannot be excluded by our results.
\end{rem}


\begin{cor}\label{cor7}
Under the conditions Theorem \ref{T_mix_fo1} suppose that one of the following conditions is satisfied.
\begin{enumerate}[(A)]
\item
The functions $g_{1}, \ldots, g_{n}$, as higher order derivations, have the same order.
\item
The functions $f_{1}, \ldots, f_{n}$ as higher order derivations, have the same order.
\end{enumerate}
Then for all $i=1, \ldots, n$ there exists a complex number $\lambda_{i}$ such that
\[
g_i(x)=\la_i x\qquad (x\in \mathbb{F})
\]
and equation \eqref{Eq_mixed} is then of the following form
 \[
  \sum_{i=1}^{n}\la_i^{N-i}f_i(x^{i})x^{N-i}=0,
 \]
where some $\la_i\in \mathbb{C}$ can be $0$.
In this case we have  $f_i\in \mathscr{D}_{n-1}(\mathbb{F})$ for all $i=1, \ldots, n$ as it was shown in \cite{EbaRieSah, GseKisVin18}.
\end{cor}

\begin{cor}\label{cor9}
 Under the conditions Theorem \ref{T_mix_fo1}, suppose that
 \[
 f_i(x)=c_i\cdot g_i(x)
 \qquad
 \left(x\in \mathbb{F}\right)
 \]
 holds for all  for all $i\in \{1, \dots, n\}$ with some nonzero constants $c_i\in \mathbb{C}$, $i=1, \ldots, n$.
 Then
 \[
 f_i(x)=\la_i x \qquad  \left(x\in \mathbb{F}\right),
 \]
 and hence
 \[
 g_i(x)=c_i\la_i x \qquad (x\in \mathbb{F})
\]
with some complex constants $\la_i$ for all $i=1, \ldots, n$. Further these constants also have to fulfill $\displaystyle\sum_{i=1}^n c_i^{q_i}\la_i^{q_i+1}=0$.
\end{cor}

\begin{rem}
It seems that all of the previously mentioned examples lead to the equation $$\sum_{i=1}^{n}\la_i^{N-i}f_i(x^{i})x^{N-i}=0.$$ We note that the class of solutions of equation \eqref{Eq_mixed} is wider in general.

The simplest example is the following. Let  $p_1, q_1, p_1-1, q_1+1$ be distinct positive integers with $p_1>1$ and $d\colon \mathbb{F}\to \mathbb{C}$ be a derivation. Define the functions $f_{1}, f_{2}, g_{1}, g_{2}$ by $f_1=g_1=g_2=d$ and
\[f_2(x)=-p_1 x
\qquad
\left(x\in \mathbb{F}\right).
\]
Then
\begin{multline*}
f_1(x^{p_1})(g_1(x))^{q_1}+f_2(x^{p_1-1})g_2(x^{q_1+1})=
d(x^{p_{1}})d(x)^{q_{1}}+(-p_{1}x^{p_{1}-1})d(x)^{q_{1}+1}
\\=
p_{1}x^{p_{1}-1}d(x)d(x)^{q_{1}}+(-p_{1}x^{p_{1}-1})d(x)^{q_{1}+1}=0
\end{multline*}
for all $x\in \mathbb{F}$.
\end{rem}

The corollaries and the remark above motivate the following conjecture.

\begin{conj}\label{conj_mixed}
Let $n$ be a positive integer, $\mathbb{F}\subset \mathbb{C}$ be a field and
 $p_{1}, \ldots, p_{n}, q_{1}, \ldots, q_{n}$ be fixed positive integers fulfilling conditions C(i)--C(iii).
 Assume that the additive functions $f_{1}, \ldots, f_{n}, g_{1}, \ldots, g_{n}\colon \mathbb{F}\to \mathbb{C}$ satisfy equation \eqref{Eq_mixed}. Then every function is a generalized exponential polynomial function of degree at most $n-1$.
 In particular, if
 \begin{equation}\label{eq_mixred3}
  f_{i}(x)= D_i(x)
  \qquad
  \text{and}
  \qquad
  g_{i}(x)= \widetilde{D}_i(x)
  \qquad
  \left(x\in \mathbb{F}^{\times}\right)
 \end{equation}
for some derivations $D_i, \widetilde{D_i}$ ($i=1, \ldots, n$), then the order of $D_i, \widetilde{D}_i$ is at most $n-1$.
\end{conj}

Although we cannot verify the conjecture in its full generality, we can handle the case when $q_i\ge\frac{N}{2}$. We note that this condition complements the one in Theorem \ref{thm_mixed}.

\begin{thm}
Let $n$ be a positive integer, $\mathbb{F}\subset \mathbb{C}$ be a field and
 $p_{1}, \ldots, p_{n}, q_{1}, \ldots, q_{n}$ be fixed positive integers fulfilling conditions C(i)--C(iii) and $q_i\ge \frac{N}{2}$.
 Assume that the additive functions $f_{1}, \ldots, f_{n}, g_{1}, \ldots, g_{n}\colon \mathbb{F}\to \mathbb{C}$ satisfy functional equation \eqref{Eq_mixed}. Then every function $f_i$ (resp. $g_i$) is generalized exponential polynomials of degree at most $n-1$.
\end{thm}
\begin{proof}
By Lemma \ref{lem_decop_mixed}, the solutions $f_i, g_i$ of \eqref{Eq_mixed} are decomposable functions for all $i=1, \dots, n$, i.e., they are generalized exponential polynomial  functions of the form $\sum P_jm_j$ and $\sum Q_jm_j$, respectively.
If restricting the equation to the terms containing $m_j$ for fixed $j$ as in Remark \ref{rem_vegy_pol} we can prove that the degree $P_j$ is at most $n-1$, then it holds in general for the original solutions $f_i, g_i$. By the equivalence relation $\sim$ we can assume that $m_j(x)=x$. Then, by Theorem \ref{thm_KisLac}, the solutions can be seen as derivations $D_i$ and $D_j$ having the same order as the degree of $P_j$ and $Q_j$, respectively.

Now we can apply Theorem \ref{T_mix_fo1}, which implies that either there is some $f_{i_0}=D_{i_0}$ of maximal order and the corresponding $g_{i_{0}}$ is of the form  $g_{i_0}(x)=c_{i_0}\cdot x\, (x\in \mathbb{F})$, or the order of $f_i, g_i$ is at most $1$ for each $i=1, \ldots, n$. In the latter case we get the result, as $n-1\ge 1$. The preceding case is more technical.
First we restrict our attention to finitely generated subfields of $\mathbb{F}$, since the result on these restricted fields implies the statement on $\mathbb{F}$. Hence from now on we assume that $\mathbb{F}$ is finitely generated. In this case the derivations are differential operators. Now as in the proof of Theorem \ref{T_mix_fo1}, using Corollary \ref{cor:alg_ind}, we can assume that every differential operator is a function of a given derivation $d$ and hence $f_i, g_i$ ($i=1\dots, n$)
can be represented as
 \[
 f_{i}(x)=D_i(x)= \sum_{j=0}^{k_i}\lambda_{i, j}d^{j}(x)
 \qquad
 \text{and}
 \qquad
 g_{i}(x)=\widetilde{D}_i(x)= \sum_{j=0}^{l_i}\widetilde{\lambda}_{i, j}d^{j}(x)
 \qquad
 \left(x\in \mathbb{F}^{\times}\right)
\] and equation \eqref{Eq_mixed} as
\[
 \sum_{i=1}^{n}\left(\sum_{j=0}^{k_i}\lambda_{i, j}d^{j}(x^{p_{i}}) \right)
 \cdot \left(\sum_{j=0}^{l_i}\widetilde{\lambda}_{i, j}d^{j}(x)\right)^{q_{i}}=0
 \qquad
  \left(x\in \mathbb{F}^{\times}\right).
\]
As in the proof of Theorem \ref{T_mix_fo1}, homogeneity argument lead to
\begin{equation}\label{eq_mixer2}\sum_{i=1}^m \lambda_{i, k_i}(\widetilde{\lambda}_{i, l_i})^{q_i}d^{k_i}(x^{p_{i}})(d^{l_i}(x))^{q_{i}}=0,
\end{equation}
where $i=1, \dots, m$ are those indices that satisfy $K=k_i+l_iq_i$ and $K$ is the largest possible. Now we have that $i_0\in \{1,\dots, m\}$, i.e., $k_{i_0}=K$ and $l_{i_0}=0$.

We have to show that $K\le n-1$, which immediately implies the statement.Assume contrary that $K\ge n$.  In the expansion of equation \eqref{eq_mixer2} all monomials of $x, d, d^2, \dots, d^K$ have to vanish. Suppose that the first $m'$  indices satisfy that $k_i=K, l_i=0$ for all $i\in \{1, \dots, m'\}$.
This also means that $\la_{i,K}\ne 0$ and $\tla_{i,0}\ne 0$.
Take monomials
\begin{equation}\label{eq_eh}d^{K-t}(d(x))^{t}x^{N-t-1} \qquad  \textrm{ for } \qquad  t=0,\dots, n-1.\end{equation}
Note that $q_i\ge \frac{N}{2}$ immediately implies that $q_i>n$ for all $i=1, \dots, n$. Hence, if $l_i>0$, i.e., $i>m'$, then none of the terms defined in \eqref{eq_eh} can appear in the expansion of $d^{k_i}(x^{p_i})(d^{l_i}(x))^{q_i}$ if $i>m'$. Hence these terms in \eqref{eq_eh} can stem only from the expansion of
$d^{K}(x^{p_{i}})(x)^{q_{i}}$, where $i=1\dots m'$.
In this case for any $t=0, \dots, m'-1$ the term $d^{K-t}(d(x))^{t}$ stems from $d^{K}(x^{p_i})$.  By Proposition \ref{prop1}, its coefficient is $$\binom{K}{K-t}\binom{p_i}{\underbrace{1,\dots,1,}_{p_i-t+1}}.$$
As the coefficient of $d^{K-t}(d(x))^{t}x^{N-t-1}$ has to vanish
we get that
$$\sum_{i=1}^{m'}\la_{i,K}(\tla_{i,0})^{q_i} \binom{K}{K-t}\cdot\binom{p_i}{\underbrace{1,\dots,1}_{t+1}}=0 \qquad  ( t=0,\dots, m'-1)$$

This can written in the following matrix form
\[
 \begin{pmatrix}
 \binom{p_{1}}{1} & \ldots & \binom{p_{m'}}{1}\\
 \binom{p_{1}}{1,1} & \ldots & \binom{p_{m'}}{1,1} \\
 \vdots & \ddots & \vdots\\
  \binom{p_{1}}{\underbrace{1, \ldots, 1}_{m'}} & \ldots & \binom{p_{m'}}{\underbrace{1, \ldots, 1}_{m'}}
 \end{pmatrix}
 \cdot
 \begin{pmatrix}
  \lambda_{1, K}(\widetilde{\lambda}_{1, 0})^{q_1}\\
  \lambda_{2, K}(\widetilde{\lambda}_{2, 0})^{q_2}\\
  \vdots \\
  \lambda_{m', K}(\widetilde{\lambda}_{m', 0})^{q_{m'}}
 \end{pmatrix}
=
\begin{pmatrix}
 0\\
 0\\
 \vdots
 \\
 0
\end{pmatrix}
 \]
Note that none of the rows of the previous matrix is identically $0$. Indeed, there exists $p_{i'}\ge m'$ for some $i'\in\{1, \dots, m'\}$, since all $p_i$ ($i=1, \dots, m'$) are different and positive integers.
In this case it is straightforward to verify that the previous matrix equation is equivalent to
\[
 \begin{pmatrix}
 p_{1} & \ldots & p_{m'}\\
 p_{1}^2 & \ldots & p_{m'}^2 \\
 \vdots & \ddots & \vdots\\
  p_{1}^{m'} & \ldots & p_{m'}^{m'}
 \end{pmatrix}
 \cdot
 \begin{pmatrix}
  \lambda_{1, K}(\widetilde{\lambda}_{1, 0})^{q_1}\\
  \lambda_{2, K}(\widetilde{\lambda}_{2, 0})^{q_2}\\
  \vdots \\
  \lambda_{m', K}(\widetilde{\lambda}_{m', 0})^{q_{m'}}
 \end{pmatrix}
=
\begin{pmatrix}
 0\\
 0\\
 \vdots
 \\
 0
\end{pmatrix}
 \]

The matrix of this equation is a Vandermonde matrix of $p_i$. As $p_i$ are all different the equation has only trivial solutions. Thus $\lambda_{i, K}(\widetilde{\lambda}_{i, 0})^{q_i}=0$ for every $i\in \{1, \dots, m'\}$, which is a contradiction as $\lambda_{i, K}\ne 0$ and $\widetilde{\lambda}_{i, 0}\ne 0$. This contradiction shows that the maximal order $K$ of $D_i$ and $\widetilde{D}_i$ is at most $n-1$. This also finishes the proof of the theorem.

\end{proof}

\subsubsection*{Special cases of equation \eqref{Eq_mixed}}
Now we consider equations 
of the form
\begin{equation}\label{eq2mixed}
 f_{1}(x^{p_{1}})g_{1}(x)^{q_{1}}+f_{2}(x^{p_{2}})g_{2}(x)^{q_{2}}=0
 \qquad
 \left(x\in \mathbb{F}\right),
\end{equation}
where $f_{1}, f_{2}, g_{1}, g_{2}\colon \mathbb{F}\to \mathbb{C}$ denote the unknown additive functions and the parameters $p_{1}, p_{2}, q_{1}, q_{2}$ fulfill conditions C(i)--C(iii), i.e., $p_1<p_2$, $p_1+q_1=p_2+q_2$ and $p_1,p_2, q_1,q_2$ are all distinct and positive integers.
Even in this case, which can be considered as the `simplest' example (as it contains only two summands), the description of all solutions are elaborate. First, in Theorem \ref{lem2tag1} we characterize all solutions, where the corresponding functions are of the form $P(x)\cdot x$. Secondly, in Lemma \ref{lem2tag2} we consider the case when the additive solutions (that are exponential polynomials on the multiplicative group) contain more than one exponentials in their representations.

In Lemma \ref{lem_decop_mixed} we have shown that all solutions of \eqref{eq2mixed} are generalized exponential polynomial functions of the form $\sum_i P_im_i$, where $P_i$ are generalized polynomials and $m_i$ are exponential functions on $\mathbb{F}^{\times}$. In Example \ref{Exmix} we illustrated that the solutions can be sums of generalized exponential polynomials, however in all examples $g_i$'s are linear combinations of different exponential functions.
At the same time, we can concentrate on the solutions, where all $f_i, g_i$ are of the form $P_im$ and $Q_im$ for a given exponential function $m$.
Using the equivalence relation $\sim$ (see Lemma \ref{lem_equ_mixed}) we can assume that solutions are of the form $P_i(x)\cdot x$, that is, those solutions are higher order derivation or linear functions. Note that these solutions are the building blocks of the solutions in general, since by algebraic independence of exponential polynomial  functions, necessarily equation \eqref{eq2mixed} has such a solution, in any case.

\begin{thm}\label{lem2tag1}
Suppose that $f_1, f_2, g_1, g_2\colon \mathbb{F}\to \mathbb{C}$ are higher order derivations that also fulfill equation \eqref{eq2mixed} so that the parameters $p_1, p_2, q_1, q_2$ satisfy conditions C(i)--C(iii). Then all of them (as higher order derivations) is of order at most $1$ and the solutions are one of the following.
\begin{enumerate}
\item[(A)] \[
 f_{1}(x)=\la_{1,1} d(x)+\la_{1,0}x, ~ ~ f_{2}(x)=\la_{2,0}x, \qquad
 \left(x\in \mathbb{F} \right)\]
 \[g_{1}(x)= \mu_{1,1} d(x)+\mu_{1,0}x, ~ ~ g_{2}(x)= \mu_{2,1} d(x)+\mu_{2,0}x,
 \qquad
 \left(x\in \mathbb{F}\right)
\]
where $\la_{1,1}, \la_{2,0}, \mu_{1,1}, \mu_{2,1}\in \mathbb{C}$ are nonzero, $q_2=q_1+1$ (i.e, $p_2=p_1+1$) satisfying
\[p_{1}\cdot\la_{1,1}\cdot(\mu_{1,0})^{q_{1}}+\la_{2,0}\cdot (\mu_{2,1})^{q_1+1}=0.\]
Equivalently, there is a function $h(x)=d(x)+ax$ such that $$f_1(x^{p_1})=\la_{1,1}p_1h(x)x^{p_1-1},\ \  g_1(x)=\mu_{1,1}h(x)\  \textrm{ and } g_2(x)=\mu_{2,1}h(x).$$
\item[(B)]
\[
 f_{1}(x) =  \lambda_{1, 1}d(x)+ \lambda_{1, 0}x, ~ ~ f_2=\lambda_{2, 0}x
 \qquad \left(x\in \mathbb{F}\right)
\]
 \[
 g_{1}(x) = \mu_{1, 0}x, ~ ~ g_{2}(x) = \mu_{2,1}d(x)+ \mu_{2, 0}x
 \qquad
 \left(x\in \mathbb{F}\right)
\]
where 
$\la_{1,1}, \la_{2,0}, \mu_{1,0}, \mu_{2,1}\in \mathbb{C}$ are nonzero, $q_2=1$ (i.e, $p_2=N-1$) satisfying
\[p_{1}\cdot\la_{1,1}\cdot(\mu_{1,0})^{q_{1}}+\la_{2,0}\cdot \mu_{2,1}=0.\]
Furthermore, $f_1(x^{p_1})=c_1g_2(x)x^{p_1-1}, \ \ g_1(x)=c_2f_2(x)$, where $c_1\cdot c_2=-\la_{2,0}^{1-q_1}$.
\item[(C)]
\[
 f_{1}(x) =  \lambda_{1, 1}d(x)+ \lambda_{1, 0}x, ~ ~ f_2=\lambda_{2, 1}d(x)+ \lambda_{2, 0}x,
 \qquad  \left(x\in \mathbb{F}\right)\]
\[ g_{1}(x) = \mu_{1, 0}x, ~ ~ g_{2}(x) = \mu_{2, 0}x,
 \qquad
 \left(x\in \mathbb{F}\right)
\]
where $\la_{i,j}, \mu_{i,0}\in \mathbb{C}$ for all $i\in \{1,2\}, j\in \{0,1\}$, that satisfies
\[p_{1}\cdot\la_{1,1}\cdot(\mu_{1,0})^{q_{1}}+p_2\cdot\la_{2,1}\cdot (\mu_{2,0})^{q_2}=0.\]
\[\la_{1,0}\cdot(\mu_{1,0})^{q_{1}}+\la_{2,0}\cdot (\mu_{2,0})^{q_2}=0\]
In particular, $(x^{{p_2}-{p_1}})f_1(x^{p_1})=cf_2(x^{p_2})$ for some $c\in \mathbb{C}$.

We note that this case includes those, when all solutions are linear functions.
\end{enumerate}
\end{thm}

\begin{proof}
 Due to the results of Theorem \ref{T_mix_fo1}, we get that the solutions of \eqref{eq2mixed} are either of the form as in Case (A) or one of $g_i$, say $g_1$, is a linear function, and $f_1$ as a derivation is of order $K$, where $K$ is maximal.
In this case let us denote the orders of $f_2$ and $g_2$ by $k_2$ and $l_2$, respectively, so $K=k_2+l_2q_2$. Hence equation \eqref{eq2mixed} can be written
 $$(\la_{1,K}d^{K}(x^{p_1})+\dots+ \la_{1,0}x^{p_1})(\mu_{1,0}x)^{q_1}+(\la_{2,k_2}d^{k_2}(x^{p_2})+\dots+ \la_{2,0}x^{p_1})(\mu_{2,l_2}d^{l_2}(x)+\dots+\mu_{2,0}x)^{q_2} =0,$$
 where $\la_{1,K}, \la_{2,k_2}, \mu_{2,l_2}$ are nonzero by assumption. The expansion of $d^{K}(x^{p_1})$ must be covered by the expansion of $d^{k_2}(x^{p_2})(d^{l_2}(x))^{q_2}$, otherwise $\la_{1,K}=0$. This immediately implies that $l_2\le 1$. We will show that if $l_2=1$, then in Case (B) happens, and if $l_2=0$ implies Case (C).
 \begin{enumerate}
\item[(A)]
 Equation \eqref{eq2mixed} can be reformulated in case (A) as follows.
$$(\la_{1,1} d(x^{p_1})+\la_{1,0}x^{p_1})(\mu_{1,1} d(x)+\mu_{1,0}x)+\la_{2,0}x^{p_1-1}(\mu_{2,1} d(x)+\mu_{2,0}x)^{q_1+1}, \qquad
 \left(x\in \mathbb{F} \right),$$
 where we assume at least one of $\la_{1,1}, \mu_{1,1}, \mu_{2,1}$ is nonzero. This implies that none of them is 0.
 The equation above is equivalent to
 $$(p_1\la_{1,1} d(x)+\la_{1,0}x)(\mu_{1,1} d(x)+\mu_{1,0}x)+\la_{2,0}(\mu_{2,1} d(x)+\mu_{2,0}x)^{q_1+1}, \qquad
 \left(x\in \mathbb{F} \right).$$
Now we use the algebraic independence of $d(x)$ and $x$, therefore every coefficient has to vanish in the expansion of the previous equation. This also means that we can substitute other polynomially independent elements to the equation, so instead of the pair $(x,d(x))$ we can substitute the pair $(1,y)$. Hence, we get
$$p_1\la_{1,1}\mu_{1,1} \left(y+\frac{\la_{1,0}}{p_1\la_{1,1}}\right)\left(y+\frac{\mu_{1,0}}{ \mu_{1,1}}\right)^{q_1}=-\la_{2,0}\mu_{2,1}\left(y+\frac{\mu_{2,0}}{\mu_{2,1}}\right)^{q_1+1}, \qquad
 \left(y\in \mathbb{F} \right).$$
Note that Since the main coefficients has to be equal we get $p_1\la_{1,1}\mu_{1,1}=- \la_{2,0}\mu_{2,1}$. Thus we can eliminate these terms.
Introducing $a=\frac{\la_{1,0}}{p_1\la_{1,1}}, b=\frac{\mu_{1,0}}{ \mu_{1,1}}$ and $c=\frac{\mu_{2,0}}{\mu_{2,1}}$ we get that
$$(y+a)(y+b)^{q_1}=(y+c)^{q_1+1}.$$
Since the polynomials in $\mathbb{C}$ are uniquely determined by its roots, hence $a=b=c$ should hold. This immediately implies the second part of Case (A).

\item[(B)]
 Recall that $f_1$ as a derivation is of order $K$, $g_1$ is linear function, and the orders of $f_2$ and $g_2$ are $k_2$ and $l_2$, respectively, so $K=k_2+l_2q_2$.
 It is clear that  $p_1=1$ is not possible if $k_2\ne K$ (i.e., $l_2=1$). If $l_2=1$ and $p_1\ne 1$, then we get that $q_2=1$, otherwise we cannot eliminate the coefficient of $d^{k_2}(x)d^{q_2}(x)x^{p_1-2}$ in the expansion of $d^{K}(x^{p_1})$ (here we use that $p_1\ge 2$). Hence we have that $k_2+1=K$
\begin{itemize}
 \item If $K\ge 4$ (and $p_1\ge 2$), then the coefficient of $d^{K-2}(x)d^2{x}x^{p_1-2}$ cannot be eliminated.
\item If $K=3$, then the expansion of $d^3(x^{p_1})x^{q_1}$ and the expansion of $d^{2}(N-1)d(x)$. Since $p_1+q_1=p_2+q_2=N\ge 5$, if $p_1=2$, then $(d(x))^3$ can only appear in the expansion of $d^{2}(x^{N-1})d(x)$, hence $\la_{2,k_2}\cdot \mu_{2,l_2}=0$, which is a contradiction. On the other hand if $p_1\ge 3$, the $d^{3}(x)$ appears only in the expansion of $d^3(x^{p_1})x^{q_1}$ which implies $\la_{1,K}=0$, a contradiction.
\item If $K=2$, then $d^2(x)$ appears only in the expansion of $d^3(x^{p_1})x^{q_1}$ (note that $p_1\ge 2$).
\item If $K=1$, then we get that the solutions are of the form as in Case (B).
 \end{itemize}

As $q_2=1$ and $K=1$, equation \eqref{eq2mixed} in Case (B) can be written of the following form
\[
(\lambda_{1, 1}d(x^{p_1})+ \lambda_{1, 0}x^{p_1})( \mu_{1, 0}x)^{q_1}+\lambda_{2, 0}x^{N-1}(\mu_{1,1}d(x)+ \mu_{2, 0}x)=0
 \qquad
 \left(x\in \mathbb{F}\right),
\]
where at least one (and hence all) $\lambda_{1, 1},\mu_{1,1}$ is nonzero. Equivalently, we have
$$\mu_{1, 0}^{q_1}(p_1\lambda_{1, 1}d(x)+\lambda_{1, 0}x)+\lambda_{2, 0}(\mu_{2,1}d(x)+ \mu_{2, 0}x)=0
 \qquad
 \left(x\in \mathbb{F}\right).$$
 Since $p_1\lambda_{1, 1}\mu_{1,0}^{q_1}=-\la_{2,0}\mu_{2,1}$,
 introducing $a=\frac{\lambda_{1, 0}}{p_1\lambda_{1, 1}}$ and $b=\frac{\mu_{2, 0}}{\mu_{2, 1}}$ we get
 $$d(x)+ax=d(x)+bx.$$ Hence $a=b$ and we get all statements of Case (B).
\item[(C)]
 If $l_2=0$, then $K=k_2$, hence $g_2$ is a linear function as well as $g_1$. Thus the equation is of the form
 $$(\la_{1,K}d^{K}(x^{p_1})+\dots+ \la_{1,0}x^{p_1})(\mu_{1,0}x)^{q_1}+(\la_{2,K}d^{K}(x^{p_2})+\dots+ \la_{2,0}x^{p_1})(\mu_{2,0}x)^{q_2} =0,$$
 where $\la_{1,K}, \la_{2,K}, \mu_{1,0},\mu_{2,0}$ are nonzero.

 Suppose that $K\ge 2$. If $\min(p_1,p_2)\ge 2$, then the following system of equations
 $$(\mu_{1,0})^{q_1}\la_{1,K}p_1+(\mu_{2,0})^{q_2}\la_{2,K}p_2=0$$
 $$(\mu_{1,0})^{q_1}\la_{1,K}p_1(p_1-1)+(\mu_{2,0})^{q_2}\la_{2,K}p_2(p_2-1)=0$$ implies that $\mu_{1,0}x\la_{1,K}=\mu_{2,0}\la_{2,K}=0$, which is a contradiction.
 If $\min(p_1,p_2)=1$, say $p_1=1$, then $(d(x))^2$ can only appear in the expansion of $d^{K}(x^{p_2})$, hence $\la_{2,K}=0$ gives a contradiction.
 Hence we get that $K=1$ and the functions are of the form as in Case (C). The rest of the statement clearly follows by direct calculations.
\end{enumerate}
In all cases we showed that every solution (as a higher order derivation) is of order at most $1$.
\end{proof}

Now we turn to the case when more than one exponential appears in the solution. Restricted to each exponential the restricted solutions are equivalent to one of the solutions described in Case (A), Case (B) and Case (C). Hence the task is to decide which are compatible with each other. By a case-by-case argument it can be shown that only the sum of different exponentials are possible as solutions.  In the following proof we study when Case (C) can be compatible with itself (containing different exponentials) and exclude that any of the functions in the solution contain a derivation as a summand. In a similar way the other cases can be excluded, but because of the length of the argument we left it to the reader.


\begin{lem}\label{lem2tag2}
If the solution of \eqref{eq2mixed} contains more than one exponential function, then it contains exactly two. This is possible only if $p_1=2, p_2=1$  i.e., equation  \eqref{eq2mixed} is of the form
\[
f_{1}(x)g_{1}^{N-1}(x)+f_{2}(x^{2})g_{2}^{N-2}(x)=0
\qquad
\left(x\in \mathbb{F}\right).
\]
and the solutions are the following.
\[
\begin{array}{rcl}
f_{1}(x)&=& a_1(\varphi_{1}(x)+c\varphi_{2}(x))\\
f_{2}(x)&=& a_2(\varphi_{2}(x)-c^2\varphi_{1}(x)) \\
g_{1}(x)&=& b_1(\varphi_{1}(x)-c\varphi_{2}(x))\\
g_{2}(x)&=& b_2(\varphi_{1}(x)-c\varphi_{2}(x)),
\end{array}
\qquad
\left(x\in \mathbb{F}^{\times}\right),
\]
where $\varphi_1$ and $\varphi_2$ are distinct automorphisms of $\mathbb{C}$, $c\in \mathbb{C}^{\times}$  and $a_i,b_i\in \mathbb{C}^{\times}$ satisfy $a_1b_1^{N-1}=-a_2b_2^{N-2}$.
\end{lem}

\begin{proof}
Suppose that $f_i$, $g_i$ ($i=1,2$) contain terms depending on two exponentials, i.e., there are two automorphisms $\varphi_1\ne \varphi_2$ of $\mathbb{C}$ and derivations $d_1, d_2$ such that every function is of the form
$c_1 \varphi_1\circ d_1+c_2\varphi_1 +c_3  \varphi_2\circ d_2+c_4\varphi_2$, where
we assume that $\varphi_1\circ d_1\ne \varphi_2\circ d_2$, otherwise we can reduce the previous term.

Now we just investigate the case when restricting the solutions to $\varphi_1$ or to $\varphi_2$ we get Case (C) in both cases. We show that in this case all functions are the linear combination of automorphisms, i.e., they do not contain nontrivial derivations. For other pairs of cases the argument is similar, but slightly different. Those we left to the reader.

Hence we assume that $$f_1=\la_{1,1} \varphi_1\circ d_1+\la_{1,0}\varphi_1 +\la_{2,1}  \varphi_2\circ d_2+\la_{2,0}\varphi_2, \ \ g_1=\mu_{1,0} \varphi_1+\mu_{2,0}\varphi_2,$$
$$f_1=\tilde{\la}_{1,1} \varphi_1\circ d_1+\tilde{\la}_{1,0}\varphi_1 +\tilde{\la}_{2,1}  \varphi_2\circ d_2+\tilde{\la}_{2,0}\varphi_2 \ \ g_1=\tilde{\mu}_{1,0} \varphi_1+\tilde{\mu}_{2,0}\varphi_2,$$
where $\mu_{1,0},\mu_{2,0},\tilde{\mu}_{1,0},\tilde{\mu}_{2,0}$ are nonzero and for each $i=1,2$ at least one of $\la{i,0}, \la{i,1}$ is non-zero. Similar holds for $\tilde{\la}_{i,j} (i=1,2, j=1,2)$. These conditions are necessary otherwise \eqref{eq2mixed} reduces to Cases (A), (B) or (C).

If $\varphi_1\ne \varphi_2$ and $\varphi_1\circ d_1\ne \varphi_2\circ d_2$, then $\varphi_1\circ d_1, \varphi_1, \varphi_2\circ d_2,\varphi_2$ are algebraically independent over $\mathbb{C}$. Therefore we can substitute them functions by $X,Y, Z, W$, respectively. Note that in this case $\varphi\circ d_1(x^{p})=p\varphi\circ d_1(x)\cdot \varphi^{p-1}(x)=X\cdot Y^{p-1}$.

In this case equation \eqref{eq2mixed} can be reformulated as follows.
\begin{equation}\label{eqexps1}
\begin{split}
    &(p_1\la_{1,1} X Y^{p_1-1}+\la_{1,0}Y^{p_1}+p_1\la_{2,1} UV^{p_1-1}+\la_{2,0}V^{p_1})(\mu_{1,0}Y+\mu_{2,0}V)^{q_1}+\\
     &(p_2\tilde{\la}_{1,1} X Y^{p_2-1}+\tilde{\la}_{1,0}Y^{p_2}+p_2\tilde{\la}_{2,1} UV^{p_2-1}+\tilde{\la}_{2,0}V^{p_2})(\tilde{\mu}_{1,0}Y+\tilde{\mu}_{2,0}V)^{q_2}=0.
\end{split}
\end{equation}
Without loss of generality we can assume that $p_1<p_2$.
Now we take the coefficient of $X Y^{p_1-1}V^{q_1}$. As $p_1<p_2$, we have $q_1>q_2$ and hence this term appears only once with coefficient $p_1\la_{1,1}\mu_{2,0}$ which is then vanishes. Hence, $\la_{1,1}=0$ ($\mu_{2,0}\ne 0$ was assumed). Similar argument for $Y^{q_1}U V^{p_1-1}$ shows that $\tilde{\la}_{1,1}=0$.
Repeating the previous argument now for $X Y^{p_2-1}V^{q_2}$ and $Y^{q_2}U V^{p_2-1}$ implies that $\la_{2,1}=0$ and $\tilde{\la}_{2,1}=0$. Hence all $f_i$ are the linear combination of $\varphi_1$ and $\varphi_2$.

We show that $p_1=1, p_2=2$ and the rest of the statement. By substituting $V=1$  we get the following equation, we get the following equation.
\begin{equation}\label{eqexps2}
 (\la_{1,0}Y^{p_1}+\la_{2,0})(\mu_{1,0}Y+\mu_{2,0})^{q_1}=
 - (\tilde{\la}_{1,0}Y^{p_2}+\tilde{\la}_{2,0})(\tilde{\mu}_{1,0}Y+\tilde{\mu}_{2,0})^{q_2}=0.
\end{equation}
This is a polynomial equation in $Y$ over $\mathbb{C}$. First we note that $P_1(Y)=\la_{1,0}Y^{p_1}+\la_{2,0}$ (and reps. $P_2(Y)=\tilde{\la}_{1,0}Y^{p_2}+\tilde{\la}_{2,0}$) has no root with multiplicity greater than 1, since $(P_i(Y), P_i'(Y))=1$ ($i=1,2$). On the other hand, $Q_1(Y)=
 (\mu_{1,0}Y+\mu_{2,0})^{q_1}$ (and  resp. $Q_2(Y)=(\tilde{\mu}_{1,0}Y+\tilde{\mu}_{2,0})^{q_2}$) has only one root with multiplicity $q_1$ (resp. $q_2$). These immediately implies that $q_2+1=q_1$ (noting that $p_1<p_2$ and $p_1+q_1=p_2+q_2$). This means that one root of $P_2(Y)$ is the same as the root of $Q_1(Y)$ and the root of $Q_2(Y)$. The other roots are the same as the roots of $P_1(Y)$. However, the roots of $P_1(Y)$ and $P_2(Y)$ are constant multiples of $p_1$'th and $(p_1+1)=p_2$'th roots of unities, respectively. They can be equal only if $p_2=2$, hence $p_1=1$. Thus $q_1=N-1, q_2=N-2$.

 In this case equation \eqref{eqexps2} is of the following form
 \begin{equation*}\label{eqexps3}
 (\la_{1,0}Y+\la_{2,0})(\mu_{1,0}Y+\mu_{2,0})^{N-1}=
 - (\tilde{\la}_{1,0}Y^2+\tilde{\la}_{2,0})(\tilde{\mu}_{1,0}Y+\tilde{\mu}_{2,0})^{N-2}=0,
\end{equation*}
where $\frac{\mu_{2,0}}{\mu_{1,0}}=\frac{\tilde{\mu_{2,0}}}{\tilde{\mu}_{1,0}}.$
Introducing $c_1=\frac{\la{2,0}}{\la_{1,0}}, c_2=\frac{\mu_{2,0}}{\mu_{1,0}}=\frac{\tilde{\mu}_{2,0}}{\tilde{\mu}_{1,0}}, c_3=\frac{\tilde{\la}_{2,0}}{\tilde{\la}_{1,0}}$ we get
$$\la_{1,0}\mu_{1,0}^{N-1}(Y+c_1)(Y+c_2)=-\tilde{\la}_{1,0}\tilde{\mu_{1,0}}^{N-2}(Y^2+c_3).$$
Hence $\la_{1,0}\mu_{1,0}^{N-1}=-\tilde{\la}_{1,0}\tilde{\mu}_{1,0}^{N-2}$ and $c_1=-c_2$ and $c_3=-c_1^2$.
This means that all solutions of \eqref{eq2mixed} that contains two exponentials are
$$f_1=a_1(\varphi_1+c\varphi_2), g_1=b_1(\varphi_2-c\varphi_2), f_2=a_2(\varphi_1-c^2\varphi_2),  g_2=b_2(\varphi_2-c\varphi_2),$$
where $c\in \mathbb{C}$ is arbitrary and $a_i,b_i\in \mathbb{C}$ satisfies $a_1b_1^{N-1}=-a_2b_2^{N-2}$ as we stated.

It is simple to prove using the previous result that there is no solution containing three exponentials.
\end{proof}

In the following two special cases are presented as illustrations of our results. 

\begin{cor}
 Let $N$ be a positive integer, $\mathbb{F}\subset \mathbb{C}$ be a field and  $p, q$ be different positive integers (strictly) less than $N$ and assume that $q\neq N-p$, let further $\kappa$ be a nonzero complex number. Suppose that the additive functions $f, g\colon \mathbb{F}\to \mathbb{C}$ are not identically zero and 
 satisfy
 \[
  f(x^{p})g(x)^{N-p}=\kappa f(x^{q})g(x)^{N-q}
  \qquad
  \left(x\in \mathbb{F}\right),
 \]
then one of the following alternatives are possible
\begin{enumerate}
    \item[(A)] there exist a derivation $d\colon \mathbb{F}\to \mathbb{C}$ and nonzero complex constants $\lambda_{0}, \lambda_{1}, \mu_{0}$
such that
\[
 f(x) \sim \lambda_{1}d(x)+\lambda_{0}x
 \qquad
 \text{and}
 \qquad
 g(x) \sim \mu_{0}x
 \qquad
 \left(x\in \mathbb{F}\right),
\]
where the above constants have even fulfill that
\[
 \lambda_{0}\mu_{0}^{N-p}(1-\kappa \mu_{0}^{p-q})=0
 \qquad
 \text{and}
 \qquad
 \lambda_{1}\mu_{0}^{N-p}(p-\kappa q \mu_{0}^{p-q})=0
\]
\item[(B)] there exist nonzero complex constants $\lambda_{0}, \mu_{0}$
such that
\[
 f(x) \sim \lambda_{0}x
 \qquad
 \text{and}
 \qquad
 g(x) \sim \mu_{0}x
 \qquad
 \left(x\in \mathbb{F}\right),
\]
where the above constants have even fulfill that
\[
 \lambda_{0}\mu_{0}^{N-p}(1-\kappa \mu_{0}^{p-q})=0
\]
\item[(C)]
\[
f(x)=a(\varphi_1(x)-\varphi_2(x))
 \qquad
 \text{and}
 \qquad
 g(x)=b(\varphi_1(x)+\varphi_2(x)),\]
where $\varphi_1, \varphi_2\colon \mathbb{F}\to \mathbb{C}$ are arbitrary but distinct field homomorphisms, $a,b\in \mathbb{C}^{\times}$ such that $a-\kappa b=0$ and if $p_1<p_2$, then $p_1=1, p_2=2$.
\end{enumerate}
\end{cor}

\begin{cor}
 Let $N$ be a positive integer, $\mathbb{F}\subset \mathbb{C}$ be a field and  $p, q$ be different positive integers (strictly) less than $N$ and assume that $q\neq N-p$, let further $\kappa$ be a nonzero complex number. If the additive functions $f, g\colon \mathbb{F}\to \mathbb{C}$ 
 satisfy
 \[
  f(x^{p})f(x)^{N-p}=\kappa g(x^{q})g(x)^{N-q}
  \qquad
  \left(x\in \mathbb{F}\right),
 \]
then there exist nonzero complex constants $\lambda_{0}, \mu_{0}$
such that
\[
 f(x) \sim \lambda_{0} x
 \qquad
 \text{and}
 \qquad
 g(x) \sim \mu_{0}x
 \qquad
 \left(x\in \mathbb{F}\right).
\]
Further $\kappa =1$, otherwise $f$ and $g$ are identically zero.
\end{cor}

\section*{Appendix}

\begin{proof}[Proof of Lemma \ref{lemma_homogenization}]
 Let $n$ be a positive integer, $\mathbb{F}\subset \mathbb{C}$ be a field and
 $p_{1}, \ldots, p_{n}, q_{1}, \ldots, q_{n}$ be fixed positive integers.

Assume that the additive functions $f_{1}, \ldots, f_{n}, g_{1}, \ldots, g_{n}\colon \mathbb{F}\to \mathbb{C}$ satisfy functional equation \eqref{Eq_mixed}
 for each $x\in \mathbb{F}$.
 Assume further that the set $\left\{ p_{1}, \ldots, p_{n}\right\}$ has a partition $\mathcal{P}_{1}, \ldots, \mathcal{P}_{k}$ with the property
 \[
 \text{if } p_{\alpha}, p_{\beta} \in \mathcal{P}_{j} \text{ for a certain index $j$, then } p_{\alpha}+q_{\alpha}= p_{\beta}+q_{\beta}.
 \]
 Observe that for all $i=1, \ldots, n$, the mapping
 \[
  \mathbb{F}\ni x\longmapsto f_{i}(x^{p_{i}})g_{i}(x))^{q_{i}}
 \]
is a generalized monomial of degree $p_{i}+q_{i}$. Indeed, it is the diagonalization of the symmetric
$(p_{i}+q_{i})$-additive mapping
\[
 \mathbb{F}^{p_{i}+q_{i}} \ni (x_{1}, \ldots, x_{p_{i}+q_{i}})
 \longmapsto
 f_{i}(x_{\sigma(1)}\cdots x_{\sigma(p_{i})})g_{i}(x_{\sigma(p_{i}+1)}
 \cdots g_{i}( x_{\sigma(p_{i}+q_{i})}).
\]
Since $\mathbb{F}\subset \mathbb{C}$, we necessarily have $\mathbb{Q}\subset \mathbb{F}$.
Let now $r\in \mathbb{Q}$ be arbitrary and substitute $rx$ in place of $x$ in equation \eqref{Eq_mixed} to get
\[
 \sum_{i=1}^{n}f_{i}((rx)^{p_{i}})g_{i}(rx)^{q_{i}}= 0
 \qquad
 \left(r\in \mathbb{Q}, x\in \mathbb{F}\right).
\]
Using the $\mathbb{Q}$-homogeneity of the additive functions $f_{1}, \ldots, f_{n}$ and $g_{1}, \ldots, g_{n}$, we deduce
\begin{multline*}
 0=
 \sum_{i=1}^{n}f_{i}((rx)^{p_{i}})g_{i}(rx)^{q_{i}}=
 \sum_{i=1}^{n}f_{i}(r^{p_{i}}x^{p_{i}})(r g_{i}(x))^{q_{i}}=
 \sum_{i=1}^{n}r^{p_{i}+q_{i}}f_{i}(x^{p_{i}})g_{i}(x)^{q_{i}}
 \\
 =
 \sum_{j=1}^{k} \sum_{p_{\alpha}\in \mathcal{P}_{j}}r^{p_{\alpha}+q_{\alpha}}  f_{\alpha}(x^{p_{\alpha}})g_{\alpha}(x)^{q_{\alpha}}
 \qquad
 \left(r\in \mathbb
 Q, x\in \mathbb{F}\right).
\end{multline*}
Note that the right hand side of this equation is a (classical) polynomial in $r$ which is identically zero. Thus all of its coefficients should be (identically) zero, yielding that the system of equations
\[
 \sum_{p_{\alpha}\in \mathcal{P}_{j}} f_{\alpha}(x^{p_{\alpha}})g_{\alpha}(x))^{q_{\alpha}}=0
 \qquad
 \left(x\in \mathbb{F}, j=1, \ldots, k\right)
\]
is fulfilled.
\end{proof}

\begin{proof}[Proof of Lemma \ref{lem_sym_mixed}]
 Let $n$ be a positive integer, $\mathbb{F}\subset \mathbb{C}$ be a field and
 $p_{1}, \ldots, p_{n}, q_{1}, \ldots, q_{n}$ be fixed positive integers fulfilling conditions C(ii).
 Assume that the additive functions $f_{1}, \ldots, f_{n}, g_{1}, \ldots, g_{n}\colon \mathbb{F}\to \mathbb{C}$ satisfy functional equation
 \[
  \sum_{i=1}^{n}f_{i}(x^{p_{i}})g_{i}(x)^{q_{i}}= 0
 \]
 for each $x\in \mathbb{F}$. Due to the additivity of the functions
 $f_{1}, \ldots, f_{n}$ and $g_{1}, \ldots, g_{n}$ for all $i=1, \ldots, n$, the mapping
 \[
  x\longmapsto f_{i}(x^{p_{i}})g_{i}(x)^{q_{i}}
 \]
 is a monomial of degree $p_{i}+q_{i}=N$. Further, it is the trace of the symmetric and $N$-additive mapping
\[
F_{i}(x_{1}, \ldots, x_{N})
= \frac{1}{N!} \sum_{\sigma \in \mathscr{S}_{N}} f_{i}(x_{\sigma(1)} \cdots x_{\sigma(p_{i})}) \cdot g_{i}(x_{\sigma(p_{i}+1)} \cdots g_{i}(x_{\sigma(N)})
 \qquad
 \left(x_{1}, \ldots, x_{N}\in \mathbb{F}\right).
\]
Therefore, the statement follows from Lemma \ref{lem_monom}.
\end{proof}

\begin{ackn}
The research of E.~Gselmann has been supported by project no.~K134191 that has been implemented
with the support provided by the National Research, Development and Innovation Fund of Hungary,
financed under the K\_20 funding scheme.

The research of G.~Kiss has been supported by projects no.~K124749  and no.~K142993 of the National Research, Development and Innovation Fund of Hungary. The author have supported by János Bolyai Research Fellowship of the Hungarian Academy of Sciences and ÚNKP-22-5 New National Excellence Program of the Ministry for Culture and
Innovation.
\end{ackn}

\bibliographystyle{plain}

\end{document}